\documentclass[11pt, a4paper]{article}

\usepackage{amsmath, amssymb, amsthm, mathrsfs, latexsym}

\usepackage{graphicx}
\usepackage{tikz}
\usetikzlibrary{calc}
\usepackage{booktabs, caption}
\usepackage{enumerate}
\usepackage{authblk} 
\usepackage{cite}    
\usepackage{CJK}    

\usepackage{xcolor}        
\usepackage[normalem]{ulem}

\usepackage[colorlinks, linkcolor=blue, anchorcolor=green, citecolor=magenta]{hyperref}
\usepackage{cleveref}

\definecolor{Blue}{rgb}{0,0,1}
\definecolor{Red}{rgb}{1,0,0}
\definecolor{DarkGreen}{rgb}{0,0.6,0}
\definecolor{DarkYellow}{rgb}{1,1,0.2}
\definecolor{DarkPurple}{rgb}{.6,0,1}

\long\def\delete#1{}

\crefformat{section}{\S#2#1#3}
\crefformat{subsection}{\S#2#1#3}
\crefformat{subsubsection}{\S#2#1#3}
\crefrangeformat{section}{\S\S#3#1#4 to~#5#2#6}
\crefmultiformat{section}{\S\S#2#1#3}{ and~#2#1#3}{, #2#1#3}{ and~#2#1#3}
\def\ma{\mathcal{A}}
\def\mb{\mathcal{B}}
\def\mc{\mathcal{C}}

\def\mf{\mathcal{F}}
\def\mg{\mathcal{G}}
\def\mh{\mathcal{H}}
\def\mi{\mathcal{I}}

\def\mm{\mathcal{M}}

\def\mt{\mathcal{T}}

\def\mw{\mathcal{W}}

\def\bs{\setminus}

\numberwithin{equation}{section}

\newtheoremstyle{mystyle}
{3pt}
{3pt}
{\itshape}
{}
{\bfseries}
{.}
{0.5em}
{\thmname{#1}\thmnumber{ #2}\thmnote{ {\normalfont(#3)}}}

\theoremstyle{mystyle}
\newtheorem{thm}{Theorem}[section]

\newtheorem{lem}[thm]{Lemma}

\newtheorem{cor}[thm]{Corollary}
\newtheorem{pr1}[thm]{Proposition}

\theoremstyle{remark}

\newtheorem{con}{\bf Construction}
\newtheorem{as}{\bf Assumption}

\begin{document}
	
		\setcounter{page}{1}
	\renewcommand{\thefootnote}{}
	\newcommand{\remark}{\vspace{2ex}\noindent{\bf Remark.\quad}}
	\renewcommand{\abovewithdelims}[2]{%
		\genfrac{[}{]}{0pt}{}{#1}{#2}}

	
	\def\qed{\hfill$\Box$\vspace{11pt}}
	
	\title {\bf   Non-trivial cross-$t$-intersecting families for vector spaces with the maximum sum of sizes}

	\author{Dehai Liu\thanks{E-mail: \texttt{liudehai@mail.bnu.edu.cn}}\   \textsuperscript{a}}
	\author{Kaishun Wang\thanks{ E-mail: \texttt{wangks@bnu.edu.cn}}\ \textsuperscript{a}}
	\author{Tian Yao\thanks{Corresponding author. E-mail: \texttt{tyao@hist.edu.cn}}\ \textsuperscript{b}}
	\affil{ \textsuperscript{a} Laboratory of Mathematics and Complex Systems (Ministry of Education), School of
		Mathematical Sciences, Beijing Normal University, Beijing 100875, China}
	
	\affil{ \textsuperscript{b} School of Mathematical Sciences, Henan Institute of Science and Technology, Xinxiang 453003, China}
	\date{}
	
	\openup 0.5\jot
	\maketitle

\begin{abstract}
	
	 Let $V$ be an $n$-dimensional vector space over a finite field. Suppose that  $\mathcal{F}$ and $\mathcal{G}$ are non-empty families of $k$-subspaces and
	 $\ell$-subspaces of $V$, respectively.
	 They are said to be cross-$t$-intersecting if $\dim(F\cap G)\geq t$ for any $F\in\mathcal{F}$ and $G\in \mathcal{G}$, and  are  further called non-trivial if $\dim(\cap_{F\in\mathcal{F}}F)<t$ and $\dim(\cap_{G\in\mathcal{G}}G)<t$.  In this paper, we characterize the  non-trivial cross-$t$-intersecting families with the maximum sum of sizes.  When $t=1$, our result serves as the $q$-analog of the theorems in \cite{F-2024-15,FT-1998-247}.

	\vspace{2mm}
	
	\
	
	
\end{abstract}

\section{Introduction}

Intersection problems have long been a central  topic in extremal combinatorics. A family of $k$-subsets of an $n$-set is called \textit{$t$-intersecting} if any two members share at least $t$ elements. The Erd\H{o}s--Ko--Rado theorem \cite{EKR-1961-313, F-1978-365, W-1984-247} states that if $n > (t+1)(k-t+1)$, then every  largest $t$-intersecting family consists of all $k$-subsets containing a fixed $t$-subset. Furthermore, the maximum-sized $t$-intersecting families whose  members have no common $t$-subset were characterized by the Hilton--Milner--Frankl theorem \cite{AK-1996-121, F-1978-146, HM-1967-369}. These classical results for finite sets have natural extensions to other mathematical structures, such as vector spaces. We refer the reader to \cite{BBCFMPT-2010-71, CLWZ-2022-1823, FW-1986-228, H-1975-1, T-2006-903, WXY-2024-220, WXZ-2023-105766} for more details.

The notion of a $t$-intersecting family can be generalized.
Two non-empty families   $\mathcal{F}$ and $\mathcal{G}$, consisting    respectively of $k$-subsets and
$\ell$-subsets of the same $n$-set,   are said to be  \textit{cross-$t$-intersecting} if $\left|F\cap G\right|\geq t$ for any $F\in\mathcal{F}$ and $G\in\mathcal{G}$. Clearly, a family $\mathcal{F}$ is $t$-intersecting if and only if $\mathcal{F}$  is cross-$t$-intersecting with itself. 
 In \cite{FT-1992-87, HM-1967-369,  WZ-2013-129}, the cross-$t$-intersecting families with the maximum sum of  sizes were characterized. Apart from a few  parameters, the extremal families $\mf$ and $\mg$  satisfy $\left|\cap_{F\in\mf}F\right|\geq t$ or $\left|\cap_{G\in\mg} G\right|\geq t$. The problem of determining  the cross-$t$-intersecting families $\mf$ and $\mg$ maximizing $|\mf| + |\mg|$, subject to $\left|\cap_{F\in\mf} F\right| < t$ and $\left|\cap_{G\in\mg} G\right| < t$, was resolved  for $t=1$ in \cite{F-2024-15, FT-1998-247}. To the best of our knowledge, this problem remains open for $t \geq 2$.

From now on, let $q$ be a prime power, and  $V$  an $n$-dimensional vector space over the finite field $\mathbb{F}_{q}$. 
Denote the family of all $k$-dimensional subspaces of $V$ by ${V\brack k}$.
Two non-empty families $\mathcal{F}\subseteq {V\brack k}$ and $\mathcal{G}\subseteq{V\brack \ell}$ are called \textit{cross-$t$-intersecting} if $\dim(F\cap G)\geq t$ for any $F\in\mathcal{F}$ and $G\in\mathcal{G}$. Wang and Zhang \cite{WZ-2013-129} completely determined the extremal cross-$t$-intersecting families with the maximum sum of  sizes. We say that   cross-$t$-intersecting families $\mathcal{F}$ and $\mathcal{G}$ are \textit{non-trivial} if $\dim(\cap_{F\in\mf} F)<t$ and $\dim(\cap_{G\in\mg} G)<t$.  Notably, an alternative notion of non-triviality   was  also   considered in \cite{CLLW-2023-105688, WL-2026-106127}.

  In this paper, we characterize  the non-trivial cross-$t$-intersecting families attaining the maximum sum of sizes.  
   Before presenting our  result, we introduce some  families. 
For subspaces $A$,  $B$ and $C$ of $V$, write 
$$\mh(A, B;\ell, t)=\left\{H\in{V\brack \ell}: \dim(H\cap A)\geq t,\ \dim(H\cap B)\geq t\right\},$$
$$\mm(C;\ell,t)=\left\{M\in{V\brack \ell}: \dim(M\cap C)\geq t+1\right\}.$$	
\begin{con}\label{2604171}
	Let  $n$, $k$, $\ell$ and $t$ be positive integers with $\ell \geq k\geq t+1$ and $n\geq k+\ell-t+1$. Suppose $A, B\in{V\brack k}$ with $\dim(A\cap B)=t-1$. 
	Then
	$$\mf_{1}=\left\{A, B\right\} \ \ \textnormal{and}  \ \  \mg_{1}=\mh(A, B;\ell, t)$$
	are non-trivial cross-$t$-intersecting families.
\end{con}	

\begin{con}\label{2604172}
	Let  $n$, $k$, $\ell$ and $t$ be positive integers with  $\min\{k,\ell\}\geq t+1$ and $n\geq k+\ell-t+1$. For $C\in{V\brack k+1}$, families
   	$$\mf_{2}={C\brack k}\ \ \textnormal{and}  \ \  \mg_{2}=\mm(C; \ell,t)$$
  are non-trivial cross-$t$-intersecting families. For $D\in{V\brack \ell+1}$, families 
  	$$\mf_{3}=\mm(D;k,t)\ \ \textnormal{and}  \ \  \mg_{3}={D\brack \ell}$$
  	are also non-trivial cross-$t$-intersecting families.
\end{con}

By symmetry, we may assume that $\ell \geq k$. Note that there are no non-trivial cross-$t$-intersecting families if $k \leq t$, and the families ${V\brack k}$ and ${V\brack\ell}$ are cross-$t$-intersecting  if $n\leq k+\ell-t$. Therefore, we restrict  attention to the case where $k \geq t+1$ and $n\geq k+\ell-t+1$.

\begin{thm}\label{1} Let $n$,  $k$,  $\ell$ and $t$ be positive integers.
	Suppose that \(\mathcal F\subseteq {V\brack k}\) and
	\(\mathcal G\subseteq {V\brack \ell}\) are non-trivial
	cross-$t$-intersecting families with the maximum $\left|\mf\right|+\left|\mg\right|$. Then 
	 $(\mathcal F,\mathcal G)$ is
	isomorphic to $(\mf_{1},\mg_{1})$, $(\mg_{1},\mf_{1})$, $(\mf_{2},\mg_{2})$ or $(\mf_{3},\mg_{3})$ under some mild lower bounds of $n$,
	as detailed in the five cases of Table~\ref{tab:classification}.
\end{thm}

\begin{center}
	\captionof{table}{Extremal non-trivial cross-$t$-intersecting families}
	\label{tab:classification}
	\renewcommand{\arraystretch}{1.15}
	\begin{tabular}{|c|c|c|c|c|}
		\hline
		Case & $t$ & $k,\ell$ & Lower bound of $n$ & $(\mathcal F,\mathcal G)$ \\
		\hline
		(i)
		& $1$
		& $\ell=k\geq 3$
		& $2k+3\delta_{2,q}+\delta_{3,q}$
		& $(\mf_{1},\mg_{1})$ or $(\mg_{1},\mf_{1})$\\
		\hline
		(ii)
		& $1$
		& $\ell>k\geq 2$
		& $ k+\ell+1+2\delta_{2,q}-\delta_{2,q}\delta_{2,k}$
		& $(\mf_{1},\mg_{1})$ \\
		\hline
		(iii)
		& $\geq 1$
		& $\ell=k=t+1$
		& $ t+2$
		& $(\mf_{2},\mg_{2})$ \\
		\hline
		(iv)
		& $\geq 2$
		& $\ell>k=t+1$
		& $\ell+t+2$
		& $(\mf_{2},\mg_{2})$ \\
		\hline
		(v)
		& $\geq 2$
		& $\ell\geq k\geq t+2$
		& $ k+\ell-t+1+3\delta_{2,q}+\delta_{3,q}+\delta_{2,q}\delta_{2,t}$
		& $(\mf_{2},\mg_{2})$ or $(\mf_{3},\mg_{3})$ \\
		\hline
	\end{tabular}
\end{center}

Both $(\mf_{2},\mg_{2})$ and $(\mf_{3},\mg_{3})$
in  Theorem \ref{1} (v) can be extremal structures.  Suppose  that $t=2$, $k=4$ and $\ell=5$. Then $(\mf_{2},\mg_{2})$ and $(\mf_{3},\mg_{3})$  attain the larger sum of  sizes for $n=9$ and $n=8$, respectively.

The rest of this paper is organized as follows. In Section \ref{2606015}, we establish some auxiliary results.   Theorem \ref{1} is  proved in Section \ref{2606016}.  We discuss cross-$t$-intersecting families under a  condition weaker than non-triviality  in Section \ref{2606019}.   For readability, the proofs of technical inequalities are deferred to Section \ref{26060110}.

\section{Preliminaries}\label{2606015}

This section collects several auxiliary results that will be used in the proofs of the main theorem.  Recall that for integers $a \geq b \geq 0$, the Gaussian binomial coefficient is defined as
$$
{a\brack b}_q
=
\prod_{i=0}^{b-1}\frac{q^{a-i}-1}{q^{b-i}-1}.
$$
 By convention, set ${a\brack b}_{q}=0$ if $b<0$ or $a<b$. In what follows,  write ${a\brack b}$ for short.

For  a family $\mf$ and a subspace $H$,  denote the family  of all members in $\mf$ containing $H$ by $\mf_{H}$. From now on, lowercase letters   denote positive integers, unless otherwise stated.

\begin{lem}[{\normalfont[\citen{WL-2026-106127}, Lemma 2.2]}]\label{2604155}
	Let  $\min\{k,  \ell\}\geq t$ and $n\geq k+\ell-t+1$. Suppose that $\mf\subseteq {V\brack k}$ and $G\in {V\brack \ell}$ satisfy $\dim(F\cap G)\geq t$ for any $F\in\mf$. If $H$ is a subspace of $V$ with $\dim(H\cap G)<t$, then there exists a subspace $U$ of $V$ with $H\subseteq U$ and $\dim(U)=\dim(H)-\dim(H\cap G)+t$ such that  
	$\left|\mf_{H}\right|\leq {\ell-\dim(H\cap G)\brack t-\dim(H\cap G)}\left|\mf_{U}\right|$.
\end{lem}

It is worth noting an immediate implication of Lemma \ref{2604155}. By
$$\frac{q^{\ell-\dim(H\cap G)}-1}{q^{t-\dim(H\cap G)}-1}\leq\frac{q^{\ell-\dim(H\cap G)-1}-1}{q^{t-\dim(H\cap G)-1}-1} \leq \cdots\leq \frac{q^{\ell-t+1}-1}{q-1},$$
we have ${\ell-\dim(H\cap G)\brack t-\dim(H\cap G)}\leq {\ell-t+1\brack 1}^{t-\dim(H\cap G)}={\ell-t+1\brack 1}^{\dim(U)-\dim(H)}$.
It follows that  $$\left|\mf_{H}\right|\leq {\ell-t+1\brack 1}^{\dim(U)-\dim(H)}\left|\mf_{U}\right|.$$
Indeed, we shall often use the above inequality in place of $\left|\mf_{H}\right|\leq {\ell-\dim(H\cap G)\brack t-\dim(H\cap G)}\left|\mf_{U}\right|$.

 For each family $\mathcal F$ of subspaces of dimension at least $t$, a \textit{$t$-cover} of $\mathcal F$ is a subspace $T$ satisfying $\dim(T\cap F)\geq t$ for any $F\in\mathcal{F}$, and the \textit{$t$-covering number} $\tau_t(\mathcal F)$ of $\mf$ is the minimum  dimension of a $t$-cover.

\begin{pr1}[{\normalfont[\citen{WL-2026-106127}, Proposition 2.3]}]\label{2604154}
	Let  $\min\{k,  \ell\}\geq t$ and $n\geq k+\ell-t+1$. If $\mf\subseteq {V\brack k}$ and $\mg\subseteq {V\brack \ell}$ are cross-$t$-intersecting, then
	$$\left|\mf\right|\leq {\tau_{t}(\mf)\brack t}{\ell-t+1\brack 1}^{\tau_{t}(\mg)-t}{n-\tau_{t}(\mg)\brack k-\tau_{t}(\mg)}.$$ 
\end{pr1}

The next proposition refines the estimate for the sizes of cross-$t$-intersecting families.

\begin{pr1}\label{2604156}
	Let  $\min\{k,  \ell\}\geq t$ and $n\geq k+\ell-t+1$. Suppose that $\mf\subseteq {V\brack k}$ and $\mg\subseteq {V\brack \ell}$ are cross-$t$-intersecting. Set 
	$\mathcal{M}=\left\{ M\in \mg: T\nsubseteq M \ \textnormal{for any}\ T\in \mt_{t}(\mf)\right\}$.
	Then 
	$$\left|\mm\right|\leq {\tau_{t}(\mg)\brack t}{k-t+1\brack 1}^{\tau_{t}(\mf)-t+1}{n-\tau_{t}(\mf)-1\brack \ell-\tau_{t}(\mf)-1}.$$
\end{pr1}
\begin{proof}
	If $\mm$ is empty, then the desired result is clear. Next, we assume that $\mm$ is non-empty.  Fix $T\in\mt_{t}(\mg)$. Then $\mm=\bigcup_{H\in{T\brack t}}\mm_{H}$. There exists $H_{1}\in {T\brack t}$ such that $\left|\mm_{H}\right|\leq \left|\mm_{H_{1}}\right|$ for any $H\in{T\brack t}$. It is sufficient to show $\left|\mm_{H_{1}}\right|\leq {k-t+1\brack 1}^{\tau_{t}(\mf)-t+1}{n-\tau_{t}(\mf)-1\brack \ell-\tau_{t}(\mf)-1}$.

	We claim that there exists a subspace $H_{u}$ of $V$ with $\dim(H_{u})\geq \tau_{t}(\mf)$ such that
	\begin{equation}\label{2604158}
		\left|\mm_{H_{1}}\right|\leq {k-t+1\brack 1}^{\dim(H_{u})-t}\left|\mm_{H_{u}}\right|. 
	\end{equation}
	If $\dim(H_{1})=\tau_{t}(\mf)$, then $H_{1}$ is our required subspace. If $\dim(H_{1})<\tau_{t}(\mf)$, then   $\dim(H_{1}\cap F_{1})<t$ for some  $F_{1}\in\mf$. By Lemma \ref{2604155}, there exists $H_{2}$ with $\dim(H_{2})> \dim(H_{1})$ such that
	$\left|\mm_{H_{1}}\right|\leq {k-t+1\brack 1}^{\dim(H_{2})-\dim(H_{1})}\left|\mm_{H_{2}}\right|$.
	Repeat the process above, and we finally get a chain of subspaces $H_{1}, H_{2}, \ldots, H_{u}$ with $$\dim(H_{1})<\dim(H_{2})<\cdots<\dim(H_{u-1})< \tau_{t}(\mf)\leq \dim(H_{u})$$  such that
	$$\left|\mm_{H_{i}}\right|\leq {k-t+1\brack 1}^{\dim(H_{i+1})-\dim(H_{i})}\left|\mm_{H_{i+1}}\right|,\ 1\leq i\leq u-1. $$
	It follows that this claim holds.

	Suppose  $\dim(H_{u})\geq \tau_{t}(\mf)+1$. Since  $\mm$ is non-empty, we know $\dim(H_{u})\leq \ell$.  This combining with (\ref{2604158}) and Lemma \ref{2604157} implies 
	$$\left|\mm_{H_{1}}\right|\leq {k-t+1\brack 1}^{\dim(H_{u})-t}{n-\dim(H_{u})\brack \ell-\dim(H_{u})}\leq {k-t+1\brack 1}^{\tau_{t}(\mf)-t+1}{n-\tau_{t}(\mf)-1\brack \ell-\tau_{t}(\mf)-1}.$$
	as asserted.

	Suppose $\dim(H_{u})= \tau_{t}(\mf)$. Since  $\mm$ is non-empty, we know $\mm_{H_{u}}$ is also non-empty.  Recall the definition of $\mm$. We conclude that $H_{u}$ is not a $t$-cover of $\mf$. This implies $\dim(H_{u}\cap F_{u})<t$ for some $F_{u}\in \mf$.  By Lemma \ref{2604155}, there exits subspace $H_{u+1}$ of $V$ with $\dim(H_{u+1})\geq \dim(H_{u})+1=\tau_{t}(\mf)+1$ such that $\left|\mm_{H_{u}}\right|\leq {k-t+1\brack 1}^{\dim(H_{u+1})-\dim(H_{u})}\left|\mm_{H_{u+1}}\right|$. It follows from (\ref{2604158}) that $$\left|\mm_{H_{1}}\right|\leq {k-t+1\brack 1}^{\dim(H_{u+1})-t}\left|\mm_{H_{u+1}}\right|\leq {k-t+1\brack 1}^{\dim(H_{u+1})-t}{n-\dim(H_{u+1})\brack \ell-\dim(H_{u+1})}. $$ 
	Observe that $\dim(H_{u+1})\leq \ell$ due to $\mm\neq \emptyset$. This together with Lemma \ref{2604157} yields $\left|\mm_{H_{1}}\right|\leq {k-t+1\brack 1}^{\tau_{t}(\mf)-t+1}{n-\tau_{t}(\mf)-1\brack \ell-\tau_{t}(\mf)-1}$. This finishes the proof.
\end{proof}

We proceed by stating two corollaries that collect several upper bounds used repeatedly throughout the paper. From now on, write 
\begin{align}
	f_{1}(n,k,\ell, t, x, y)
	&={x\brack t}{\ell-t+1\brack 1}^{y-t}{n-y\brack k-y}
	+{y\brack t}{k-t+1\brack 1}^{x-t}{n-x\brack \ell-x}. \label{2604193}\\
	f_{2}(n,k,\ell, t, x, y)
	&={\ell\brack t}{\ell-t+1\brack 1}^{y-t}{n-y\brack k-y}
	+{k\brack t}{k-t+1\brack 1}^{x-t}{n-x\brack \ell-x}. \label{2605193}\\
	f_{3}(n,k,\ell, t, x,y,z)
	&= {x\brack t}{\ell-t+1\brack 1}^{y-t}{n-y\brack k-y}
	+{k\brack t}{k-t+1\brack 1}^{x-t+1}{n-x-1\brack \ell-x-1}\label{2605057}\\
	&\quad\ +z{n-x\brack \ell-x}. \notag
\end{align}

\begin{cor}\label{2605191}
	Let  $\min\{k,  \ell\}\geq t$ and $n\geq k+\ell-t+1$. If $\mf\subseteq {V\brack k}$ and $\mg\subseteq {V\brack \ell}$ are cross-$t$-intersecting, then
	$\left|\mf\right|+\left|\mg\right|\leq f_{1}(n,k,\ell, t, \tau_{t}(\mf), \tau_{t}(\mg))$.  
	Moreover, if  $\tau_{t}(\mf)\geq x\geq t$ and $\tau_{t}(\mg)\geq y\geq t$, then  $\left|\mf\right|+\left|\mg\right|\leq f_{2}(n,k,\ell, t, x, y)$.  
\end{cor}
\begin{proof}
	From Proposition \ref{2604154}, we obtain $|\mf|+|\mg|\leq f_{1}(n,k,\ell,t,\tau_{t}(\mf), \tau_{t}(\mg))$. Moreover, if  $\tau_{t}(\mf)\geq x\geq t$ and $\tau_{t}(\mg)\geq y\geq t$, then by Lemma \ref{2604157}, we have $f_{1}(n,k,\ell,t,\tau_{t}(\mf), \tau_{t}(\mg))\leq f_{2}(n,k,\ell,t,x,y)$. This implies the desired result.
\end{proof}

\begin{cor}\label{2605192}
	Let  $\min\{k,  \ell \}\geq t$ and $n\geq k+\ell-t+1$. If $\mf\subseteq {V\brack k}$ and $\mg\subseteq {V\brack \ell}$ are cross-$t$-intersecting families with  $\tau_{t}(\mg)\geq y\geq t$  and $\left|\mt_{t}(\mf)\right|\leq z$, then
	$$\left|\mf\right|+\left|\mg\right|\leq f_{3}(n,k,\ell,t,\tau_{t}(\mf), y,z).$$ 
\end{cor}
\begin{proof}
By Proposition \ref{2604154} and Lemma \ref{2604157}, we get $\left|\mf\right|\leq {\tau_{t}(\mf)\brack t}{\ell-t+1\brack 1}^{y-t}{n-y\brack k-y}$. Note that 
$$\mg=\left\{ G\in \mg: T\nsubseteq G \ \textnormal{for any}\ T\in \mt_{t}(\mf)\right\}\cup\left\{G\in\mg: T\subseteq G \ \textnormal{for some}\ T\in \mt_{t}(\mf)\right\}.$$
It follows from Proposition \ref{2604156} and $\tau_{t}(\mg)\leq k$ that 
$$\left|\mg\right|\leq {k\brack t}{k-t+1\brack 1}^{\tau_{t}(\mf)-t+1}{n-\tau_{t}(\mf)-1\brack \ell-\tau_{t}(\mf)-1}+\left|\mt_{t}(\mf)\right|{n-\tau_{t}(\mf)\brack \ell-\tau_{t}(\mf)}.$$
which implies  $$\left|\mf\right|+\left|\mg\right|\leq f_{3}(n,k,\ell,t,\tau_{t}(\mf), y,\left|\mt_{t}(\mf)\right|)\leq f_{3}(n,k,\ell,t,\tau_{t}(\mf), y,z).$$ 
The required result follows.
\end{proof}

The sum $|\mf_1|+|\mg_1|$ in Construction \ref{2604171} is independent of $A$ and $B$, and is denoted by $h_{1}(n,k,\ell,t)$. Likewise, the sums $|\mf_2|+|\mg_2|$ and $|\mf_3|+|\mg_3|$ in Construction \ref{2604172} are independent of $C$ and $D$, respectively. We denote the former by $m_{1}(n,k,\ell,t)$, and then the latter is equal to $m_{1}(n,\ell,k,t)$ by symmetry. For convenience, write 
\begin{align}
	h_{2}(n,k,\ell,t)=&{k-t+1\brack 1}^{2}{n-t-1\brack \ell-t-1}-2q{k-t+1\brack 2}{k-t+1\brack 1}{n-t-2\brack \ell-t-2},\label{2604174}\\
	m_{2}(n,k,\ell,t)=&{k+1\brack t+1}{n-t-1\brack \ell-t-1}-q{t+1\brack 1}{k+1\brack t+2}{n-t-2\brack \ell-t-2}. \label{2604175}
\end{align}

\begin{lem}\label{2604173}
	Let   $\ell \geq k\geq t+1$ and $n\geq k+\ell-t+1$. Then 
	$h_{1}(n,k,\ell,t)> h_{2}(n,k,\ell,t)$.
	In particular, we have
	$$\frac{h_{1}(n,k,\ell,t)}{{k-t+1\brack 1}^{2}{n-t-1\brack \ell-t-1}}> 1-\frac{2}{(q^{2}-1)q^{n-k-\ell+t-1}}. $$
\end{lem}
\begin{proof}
	From Lemma \ref{2604282}, we obtain 
	\begin{equation*}
		\begin{aligned}
			\frac{h_{2}(n,k,\ell,t)}{{k-t+1\brack 1}^{2}{n-t-1\brack \ell-t-1}}=1-\frac{2q(q^{k-t}-1)(q^{\ell-t-1}-1)}{(q^{2}-1)(q^{n-t-1}-1)}\geq 1-\frac{2}{(q^{2}-1)q^{n-k-\ell+t-1}}.
		\end{aligned}
	\end{equation*}
	We just need to show  $h_{1}(n,k,\ell,t)> h_{2}(n,k,\ell,t)$.

	Let $A$, $B$, $\mf_{1}$ and $\mg_{1}$ be as in Construction \ref{2604171}. To prove $h_{1}(n,k,\ell,t)> h_{2}(n,k,\ell,t)$, it is sufficient to show 
	$\left|\mg_{1}\right|\geq h_{2}(n,k,\ell,t)$. 
	If $\ell=t+1$, then $k=t+1$ and 
	$$\mg_{1}=\left\{A_{1}+B_{1}: A_{1}\in{A\brack t},\ A\cap B\subseteq A_{1},\ B_{1}\in{B\brack t},\ A\cap B\subseteq B_{1}\right\}.$$
	This implies $\left|\mg_{1}\right|=(q+1)^{2}=h_{2}(n,t+1,t+1,t)$, as desired. Next, we assume $\ell\geq t+2$.
	
	In the remaining of this proof, for  $i\geq t$ and $j\geq t$, write 
	\begin{equation*}
		\begin{aligned}
			\mg_{1}(i,j)=&\left\{G\in {V\brack \ell}: A\cap B\subseteq G,\ \dim(G\cap A)=i,\ \dim(G\cap B)=j \right\},\\
			\mh(i,j)=&\left\{(A_{1},B_{1}, G)\in {A\brack i}\times{B\brack j}\times {V\brack \ell}: A\cap B\subseteq A_{1},\ A\cap B\subseteq B_{1},\ A_{1}+B_{1}\subseteq G \right\}.
		\end{aligned}
	\end{equation*}
	Adopting notations above, we have $$\mh(i,j)=\bigcup_{G\in\mg_{1}}\left\{(A_{1},B_{1}, G): A_{1}\in {A\cap G\brack i},\ A\cap B\subseteq A_{1},\ B_{1}\in {B\cap G\brack j},\ A\cap B\subseteq B_{1} \right\}.$$ 
	Indeed, the union in this equality is disjoint.  Set $$\mi=\left\{(\dim(G\cap A),\dim(G\cap B)): G\in\mg_{1},\ A\cap B\subseteq G \right\}.$$
	It is routine to check that $\{(t,t), (t,t+1), (t+1,t)\}\subseteq \mi$. We further conclude that
	$$\left|\mh(t,t)\right|=\sum_{(i,j)\in \mi}{i-t+1\brack 1}{j-t+1\brack 1}\left| \mg_{1}(i,j)\right|,$$
	$$\left|\mh(t,t+1)\right|=\sum_{(i,j)\in \mi}{i-t+1\brack 1}{j-t+1\brack 2}\left| \mg_{1}(i,j)\right|, $$
	$$\left|\mh(t+1,t)\right|=\sum_{(i,j)\in \mi}{i-t+1\brack 2}{j-t+1\brack 1}\left| \mg_{1}(i,j)\right|.$$
Note that $i\geq t$ and $j\geq t$ for any $(i,j)\in\mi$. Therefore, we have
	\begin{equation*}
		\begin{aligned}
			&\left|\mh(t,t)\right|-q\left|\mh(t,t+1)\right|-q\left|\mh(t+1,t)\right|\\
			=&\left|\mg_{1}(t,t)\right|+\left|\mg_{1}(t,t+1)\right|+\left|\mg_{1}(t+1,t)\right|			
			+\sum_{(i,j)\in\mi,\ i+j\geq 2t+2} \left({i-t+1\brack 1}{j-t+1\brack 1}\right.\\
			&\left.-q{i-t+1\brack 1}{j-t+1\brack 2}-q{i-t+1\brack 2}{j-t+1\brack 1}\right)\left|\mg_{1}(i,j)\right|. 
		\end{aligned}
	\end{equation*}
 For any  $i\geq t+2$, inequality 
	${i-t+1\brack 1}\leq q{i-t+1\brack 2}$ holds. Moreover, if $\min\{i,j\}\geq t+1$, then
	$${i-t+1\brack 1}{j-t+1\brack 1}\leq
	q{i-t+1\brack 1}{j-t+1\brack 2}+q{i-t+1\brack 2}{j-t+1\brack 1}. $$
Divide $(i,j)\in\mi$ with  $i+j\geq 2t+2$ into three cases: $i=t$, $j=t$ and $\min{i,j}\geq t+1$. Then
	$$\left|\mh(t,t)\right|-q\left|\mh(t,t+1)\right|-q\left|\mh(t+1,t)\right|\leq \left|\mg_{1}(t,t)\right|+\left|\mg_{1}(t,t+1)\right|+\left|\mg_{1}(t+1,t)\right|\leq \left|\mg_{1}\right|.$$
	
	Counting the number of pairs $(A_{1}, B_{1})$ first, we obtain
	$\left|\mh(t,t)\right|={k-t+1\brack 1}^{2}{n-t-1\brack \ell-t-1}$ and
	$$ \left|\mh(t,t+1)\right|=\left|\mh(t+1,t)\right|={k-t+1\brack 1}{k-t+1\brack 2}{n-t-2\brack \ell-t-2},$$
	implying that $$\left|\mh(t,t)\right|-q\left|\mh(t,t+1)\right|-q\left|\mh(t+1,t)\right|=h_{2}(n,k,\ell,t).$$ Hence $\left|\mg_{1}\right|\geq h_{2}(n,k,\ell,t)$, as desired.   
\end{proof}

\begin{lem}\label{2604191}
	Let  $\min\{k,\ell\}\geq t+1$ and $n\geq k+\ell-t+1$. Then $m_{1}(n,k,\ell, t)> m_{2}(n,k,\ell, t)$. In particular, we have
	$$\frac{m_{1}(n,k,\ell,t)}{{k+1\brack t+1}{n-t-1\brack \ell-t-1}}>1-\frac{1}{(q-1)q^{n-k-\ell+t}}.$$
\end{lem}
\begin{proof}
	By Lemma \ref{2604282}, we get
	$$\frac{m_{2}(n,k,\ell, t)}{{k+1\brack t+1}{n-t-1\brack \ell-t-1}}=1-\frac{q(q^{t+1}-1)(q^{k-t}-1)(q^{\ell-t-1}-1)}{(q^{t+2}-1)(q-1)(q^{n-t-1}-1)}\geq 1-\frac{1}{(q-1)q^{n-k-\ell+t}}.$$
	It is sufficient to prove $m_{1}(n,k,\ell, t)> m_{2}(n,k,\ell, t)$.
	
	Let $C$, $\mf_{2}$ and $\mg_{2}$ be as in Construction \ref{2604172}. Now, we show $\left|\mg_{2}\right|\geq m_{2}(n,k,\ell, t)$.  If $\ell=t+1$, then $\mg_{2}={C\brack t+1}$ and $\left|\mg_{2}\right|=m_{2}(n,k,t+1,t)$, as asserted. Next, we assume $\ell\geq t+2$. 
	
	To simplify the rest of this proof, we set $m=\min\{k+1, \ell\}$. For $i\geq t+1$, write
	$$\mm(i)=\left\{(C_1, G)\in {C\brack i}\times {V\brack \ell}:C_{1}\subseteq G\right\},\ \mg_{2}(i)=\left\{G\in {V\brack \ell}: \dim(G\cap C)=i\right\}.$$
	Counting  $\left|\mm(t+1)\right|$ and $\left|\mm(t+2)\right|$ in two ways yields
	\begin{equation*}
		\begin{aligned}
			\left|\mm(t+1)\right|=&{k+1\brack t+1}{n-t-1\brack \ell-t-1}=\sum_{i=t+1}^{m}{i\brack t+1}\left|\mg_{2}(i)\right|, \\
			\left|\mm(t+2)\right|=&{k+1\brack t+2}{n-t-2\brack \ell-t-2}=\sum_{i=t+2}^{m}{i\brack t+2}\left|\mg_{2}(i)\right|. 
		\end{aligned}		
	\end{equation*} 
	It follows that 
	\begin{equation*}
		\begin{aligned}
			m_{2}(n,k,\ell,t)=&\left|\mm(t+1)\right|-q{t+1\brack 1}\left|\mm(t+2)\right|\\
			=&  \left|\mg_{2}(t+1)\right|+\left|\mg_{2}(t+2)\right|+\sum_{i=t+3}^{m}\left({i\brack t+1}-q{t+1\brack 1} {i\brack t+2}\right)\left|\mg_{2}(i)\right|.
		\end{aligned}
	\end{equation*}
	Note that ${i\brack t+1}\leq q{t+1\brack 1}{i\brack t+2}$ for any $i\geq t+3$.  Therefore, we get
	$m_{2}(n,k,\ell,t)\leq \left|\mg_{2}(t+1)\right|+\left|\mg_{2}(t+2)\right|\leq \left|\mg_{2}\right|$. The desired result follows. 
\end{proof}

Recall that a family $\mf\subseteq {V\brack k} $ is said to be \textit{$t$-intersecting} if $\dim(F_{1}\cap F_{2})\geq t$ for any $F_{1},F_{2}\in\mf$. Let us close this section with a structural result for cross-$t$-intersecting families, obtained by imposing the intersection property on one  family.

\begin{pr1}\label{2605201}
	Let $\min\{k,\ell\}\geq t+1$ and $n\geq k+\ell-t+1$. Suppose that $\mf\subseteq {V\brack k}$ and $\mg\subseteq {V\brack \ell}$ are cross-$t$-intersecting. If   $\mf$ is $(k-1)$-intersecting and $\dim(\cap_{F\in\mf} F)<t$, then   $\mf\subseteq {C\brack k}$ and $\mg\subseteq \mm(C;\ell,t)$ for some $C\in{V\brack k+1}$.
\end{pr1}
\begin{proof}
	
	Since $\mf$ is $(k-1)$-intersecting, it must be contained in a maximal clique of the Grassmann graph on the $k$-dimensional subspaces of $V$. As characterized in \cite[Remark (ii), Section 9.3]{BCN-1989}, the maximal cliques in this graph fall into two classes: the collection of all $k$-dimensional subspaces containing a fixed member of ${V\brack k-1}$, and the collection of all $k$-dimensional subspaces contained in a fixed member of ${V\brack k+1}$.

	It follows from $\dim(\cap_{F\in\mf}F)<t$ that there exists $C\in{V\brack k+1}$ such that $\mf\subseteq {C\brack k}$.  As $\mf$, $\mg$ are cross-$t$-intersecting, we know $\dim(G\cap C)\geq t$ for any $G\in \mg$.   Suppose for contradiction $\dim(G\cap C)=t$ for some  $G\in\mg$. By $\dim(\cap_{F\in\mf}F)<t$, there exists $F\in\mf$ such that $G\cap C\nsubseteq F$, which implies $\dim(G\cap F)=\dim(G\cap C\cap F)<t$, a contradiction. Hence $\mg\subseteq \mm(C;\ell,t)$. The required result holds.
\end{proof}

\section{Proof of the main result}\label{2606016}

This section is devoted to the proof of Theorem \ref{1}, which is organized as follows: cases (i) and (ii) are established in Section \ref{2606075},  (iii) in Section \ref{2606076}, and the other cases in Section \ref{2606077}.

Let $\min\{k,\ell\}\geq t+1$ and $n\geq k+\ell-t+1$. Suppose that $\mf\subseteq {V\brack k}$ and $\mg\subseteq {V\brack \ell}$ are  non-trivial cross-$t$-intersecting with the maximum sum. Construction \ref{2604172} yields 
\begin{equation}\label{2606079}
	\left|\mf\right|+\left|\mg\right|\geq\max\{ m_{1}(n,k,\ell,t), m_{1}(n,\ell,k,t)\}.
\end{equation}
Furthermore, if $\ell\geq k$, then Construction \ref{2604171} guarantees that 
\begin{equation}\label{2606061}
	\left|\mf\right|+\left|\mg\right|\geq h_{1}(n,k,\ell,t).
\end{equation}

\subsection{Proof of Theorem \ref{1} (i) and (ii)}\label{2606075}

The  two  lemmas below are needed  to analyze the number of minimum $1$-covers.  For subspaces $A$, $B$ of $V$, write 
$$\mw_{1}(A, B)=\left\{A_{1}+B_{1}: A_{1}\in{A\brack 1},\ B_{1}\in {B\brack 1}\right\}, $$
$$\mw_{2}(A,B)=\left\{W\in\mw_{1}(A,B):  \dim(W\cap A \cap B)=0\right\}.$$

\begin{lem}\label{2604212}
	Let  $n\geq 2k\geq 4$. Suppose $A, B\in {V\brack k}$ with $\dim(A\cap B)=0$, and $\mt\subseteq\mw_{1}(A,B)$. If $\left|\mt\right|> {k-1\brack 1}\left({k\brack 1}+q\right)$, then the only $k$-dimensional $1$-covers of $\mt$  are  $A$ and $B$. 
\end{lem}
\begin{proof}
	It is routine to check that $A$ and $B$ are $k$-dimensional covers of $\mt$.  Next, we prove that no member in ${V\brack k}\bs\left\{A, B\right\}$ is a $1$-cover of $\mt$.   We just need to show  $$\left|\left\{W\in\mw_{1}(A,B): \dim(W\cap G)>0\right\}\right|\leq {k-1\brack 1}\left({k\brack 1}+q\right)\ \textnormal{for any} \  G\in {V\brack k}\bs\left\{A, B\right\}.$$
	
	Pick $G\in {V\brack k}\bs\left\{A, B\right\}$. In the remaining of this proof, set $\dim(G\cap A)=x$, $\dim(G\cap B)=y$, 
	$$\ma=\left\{W\in \mw_{1}(A,B): W\cap A\subseteq G\right\},$$
	$$\mb=\left\{W\in \mw_{1}(A,B): W\cap B\subseteq G\right\},$$
	$$\mc=\left\{W\in \mw_{1}(A,B): W\cap A\nsubseteq G,\ W\cap B\nsubseteq G,\ \dim(W\cap G)>0\right\}.$$
	Note that $(A_{1}+B_{1})\cap A= A_{1}$ and $(A_{1}+B_{1})\cap B= B_{1}$ for any $A_{1}+B_{1}\in\mw_{1}(A, B)$. Then 	 
	$\ma\subseteq \mw_{1}(G\cap A, B)$ and $\mb \subseteq \mw_{1}(A, G\cap B)$, which implies
	$\left|\ma\right|\leq {x\brack 1}{k\brack 1}$ and $ \left|\mb\right|\leq {y\brack 1}{k\brack 1}$. 
	
	We claim that $\left|\mc\right|\leq {k\brack1 }-{x+y\brack 1}$. If $\mc$ is empty, then there is nothing to prove. Suppose that $\mc$ is non-empty. Let $P_A, P_B: A+B \to A+B$ be the projections onto $A$ along $B$, and onto $B$ along $A$, respectively.  
	Observe that $P_{A}(U)\subseteq C$ and $P_{B}(U)\subseteq D$ for any subspaces $C$, $D$ and $U$ with $C\subseteq A$, $D\subseteq B$ and $U\subseteq C+D$. 	 
	Choose $W\in\mc$. By $W\cap A\nsubseteq G$ and $W\cap B\nsubseteq G$, we have $$\dim(W\cap G)=1, $$
	and $\dim(P_{A}(W\cap G))=\dim(P_{B}(W\cap G))=1$. Note that $W=(W\cap A)+(W\cap B)$. We further conclude $P_{A}(W\cap G)=W\cap A$ and $P_{B}(W\cap G)=W\cap B$.
	Moreover, if $W_{1}, W_{2}\in \mc$ with $W_{1}\cap G=W_{2}\cap G$, then 
	$$W_{1}=P_{A}(W_{1}\cap G)+P_{B}(W_{1}\cap G)= P_{A}(W_{2}\cap G)+P_{B}(W_{2}\cap G)=W_{2}.$$ Observe that  $$W\cap G\nsubseteq (G\cap A)+(G\cap B).$$ Otherwise  $W\cap A=P_{A}(W\cap G)\subseteq G\cap A$, a contradiction.
	Therefore, we know $\left|\mc\right|\leq\left| {G\brack 1}\bs {(G\cap A)+(G\cap B)\brack 1}\right|={k\brack 1}-{x+y\brack 1}$. This claim holds.

	From the above discussion, we deduce the following inequality.
	\begin{equation*}
		\left|\ma\right|+\left|\mb\right|+\left|\mc\right|\leq{x\brack 1}{k\brack 1}+{y\brack 1}{k\brack 1}+{k\brack 1}-{x+y\brack 1}.
	\end{equation*} 
	Since $0\leq x,y\leq k-1$,  $x+y\leq k$, by Lemma \ref{2604221}, we get $\left|\ma\right|+\left|\mb\right|+\left|\mc\right|\leq {k-1\brack 1}\left({k\brack 1}+q\right)+1$. Moreover, equality holds if and only if $(x,y)=(1,k-1)$ or $(x,y)=(k-1,1)$, and we get $\ma\cap \mb\neq \emptyset$  in  either case.  Note that $$\ma\cup\mb\cup \mc=\left\{W\in\mw_{1}(A,B): \dim(W\cap G)>0\right\}.$$ 	If $(x,y)\neq (1,k-1)$ and $(x,y)\neq (k-1,1)$, then $$\left|\ma\cup\mb\cup \mc\right|\leq \left|\ma\right|+\left|\mb\right|+\left|\mc\right|< {k-1\brack 1}\left({k\brack 1}+q\right)+1.$$ If $(x,y)=(1,k-1)$ or $(x,y)= (k-1,1)$, then by $\ma\cap \mb\neq \emptyset$, we know $$\left|\ma\cup\mb\cup \mc\right|\leq \left|\ma\cup\mb\right|+\left|\mc\right|\leq\left|\ma\right|+\left|\mb\right|+\left|\mc\right|-1<{k-1\brack 1}\left({k\brack 1}+q\right)+1.$$
   This completes the proof.
\end{proof}

\begin{lem}\label{2604241}
	Let   $n\geq 2k\geq 2s+4$. Suppose $A, B\in {V\brack k}$ with $\dim(A\cap B)=s$, and $\mt\subseteq\mw_{2}(A,B)$.  If $\left|\mt\right|> q^{2s-1}{k-s\brack 1}\left({k-s\brack 1}+q\right)$,
	then each  $k$-dimensional $1$-cover of $\mt$ contains $A\cap B$. 
\end{lem}
\begin{proof}

	We just need to show for each $G\in {V\brack k}$ with $A\cap B\nsubseteq G$, 
	\begin{equation}\label{2604252}
		\left|\left\{W\in\mw_{2}(A,B):  \dim(W\cap G)> 0\right\}\right|\leq q^{2s-1}{k-s\brack 1}\left({k-s\brack 1}+q\right).
	\end{equation}
	Indeed, for every such $G$, the family
	$$\left\{L\in{A+B\brack k}: G\cap(A+B)\subseteq L,\ A\cap B \nsubseteq L\right\}$$
	is non-empty. Therefore, it is enough to prove (\ref{2604252}) for all $G\in {A+B\brack k}$ with $A\cap B\nsubseteq G$. 
	We remark here  $\dim(W)=2$, $W\cap A=A_1$ and $W\cap B=B_1$ for each  $W=A_{1}+B_{1}\in\mw_{2}(A,B)$.

	Pick $G\in {A+B\brack k}$ with $A\cap B\nsubseteq G$. In the remaining of this proof, set $\dim(G\cap A\cap B)=w$, $\dim(G\cap A)=x$, $\dim(G\cap B)=y$,  
	$$\ma=\left\{W\in \mw_{2}(A,B): W\cap A\subseteq G\right\},$$
	$$\mb=\left\{W\in \mw_{2}(A,B): W\cap B\subseteq G\right\},$$
	$$\mc=\left\{W\in \mw_{2}(A,B):  W\cap A\nsubseteq G,\ W\cap B\nsubseteq G,\ \dim(W\cap G)>0\right\}.$$
	Note that  $w$ can be $0$.  Moreover, by $G\subseteq A+B$, we know $x\geq s$ and $y\geq s$.

	For each  $W=A_{1}+B_{1}\in\mw_{2}(A,B)$, if $W\cap A=A_{1}\subseteq G$, then $A_1\in {G\cap A\brack 1}$. Moreover, we have  $A_{1}\nsubseteq G\cap A\cap B$ and $B_{1}\nsubseteq A\cap B$. Hence
	$$\ma\subseteq \left\{A_{1}+B_{1}: A_{1}\in {G\cap A\brack 1}, \ B_{1}\in {B\brack 1},\ A_{1}\nsubseteq G\cap A\cap B,\ B_{1}\nsubseteq A\cap B\right\}.$$
	This implies $\left|\ma\right|\leq q^{w+s}{x-w\brack 1}{k-s\brack 1}$. Similarly, we have $\left|\mb\right|\leq q^{w+s}{y-w\brack 1}{k-s\brack 1}$.

	We claim that $\left|\mc\right|\leq q^{s}\left( {k\brack 1}-{x\brack 1}-{y\brack 1}+{w\brack 1}\right)$. If $\mc$ is empty, then the desired result is clear. Suppose that $\mc$ is non-empty.  Pick  $W\in \mc$. Then $\dim(W\cap G)=1$. We  further conclude  $W\cap G\nsubseteq A\cap G$. Otherwise $W\cap A=W\cap G\subseteq G$, a contradiction. Similarly, we know $W\cap G\nsubseteq B\cap G$. Hence
	$$\mc=\bigcup_{C\in {G\brack 1}\bs \left({G\cap A\brack 1}\cup {G\cap B\brack 1}\right)}\mc_{C}.$$
	Choose $C$ from the above family  such that $\left|\mc_{C}\right|$ is maximum. Since $\mc$ is non-empty, we know $\mc_{C}$ is non-empty.  For each $A_1+B_1\in \mc_{C}$, by $C\nsubseteq A$ and $C\nsubseteq B$, we know   $A_1+B_1=A_1+C=B_1+C$, which implies 
	$$A_1+B+C=B_1+B+C=B+C,\ \ 
	 B_1+A+C=A_1+A+C=A+C.$$
	Hence $A_1\subseteq B+C$ and $B_1\subseteq A+C$. Set $A_C=A\cap (B+C)$ and $B_C=B\cap (A+C)$. Then  $$\mc_{C}\subseteq \left\{ W\in {A_C+B_C\brack 2}: C\subseteq W,\ \dim(W\cap A\cap B)=0\right\}. $$
	Note that $\dim(A+C)=\dim(B+C)=k+1$. Then $\dim(A_C)=\dim(B_C)=s+1$. It follows that $\dim(A_C+B_C)=s+2$. This together with $\dim(C\cap A\cap B)=0$ yields $\left|\mc_{C}\right|\leq q^{s}$.  Therefore, we know $\left|\mc\right|\leq q^{s}\left( {k\brack 1}-{x\brack 1}-{y\brack 1}+{w\brack 1}\right)$,
	as required.

	According to the above  discussion, we have the following inequality.
	\begin{equation*}
		\left|\ma\right|+\left|\mb\right|+\left|\mc\right|\leq\ q^{w+s}{k-s\brack 1}\left({x-w\brack 1}+{y-w\brack 1}\right)+q^{s}\left( {k\brack 1}-{x\brack 1}-{y\brack 1}+{w\brack 1}\right).
	\end{equation*}
	Recall that $0\leq w\leq s-1$ and $s\leq x,y\leq k-1$.  Since $(G\cap A)+(G\cap B)$ is a subspace of  $G$, we get  $x+y-w\leq k$. From Lemma \ref{2605292}, we obtain $\left|\ma\right|+\left|\mb\right|+\left|\mc\right|\leq q^{2s-1}{k-s\brack 1}\left({k-s\brack 1}+q\right)$.  This together with $$\left\{W\in\mw_{2}(A,B): \dim(W\cap G)>0\right\}=\ma\cup\mb\cup \mc$$ 
	yields (\ref{2604252}). The proof is completed.
\end{proof}

We proceed by estimating the sum of sizes of non-trivial  cross-$1$-intersecting families in term of their $1$-covering numbers.

\begin{as}\label{2606072}
	Let $\ell\geq k\geq 2$ and $(k,\ell)\neq (2,2)$. Assume  that $n\geq 2k+3\delta_{2,q}+\delta_{3,q}$ when $\ell=k$, and $n\geq k+\ell+1+2\delta_{2,q}-\delta_{2,q}\delta_{2,k}$ when $\ell> k$. Suppose that $\mf\subseteq {V\brack k}$ and $\mg\subseteq {V\brack \ell}$ are non-trivial cross-$1$-intersecting families. 
\end{as}

\begin{lem}\label{26204271}
	Let notations be as in Assumption \ref{2606072}.   Then $\left|\mf\right|+\left|\mg\right|<h_{1}(n,k,\ell,1)$ if one of the following holds: 
	\begin{itemize}
		\item[\normalfont{(i)}]  $k=\ell$, and  $\tau_{1}(\mf)=\tau_{1}(\mg)=2$  or $\min\{\tau_{1}(\mf), \tau_{1}(\mg)\}\geq 3$;
		\item[\normalfont{(ii)}] $k<\ell$ and  $\tau_{1}(\mf)\geq 3$.
	\end{itemize}
\end{lem}
\begin{proof} 

(i) By Corollary \ref{2605191}, we have $\left|\mf\right|+\left|\mg\right|\leq f_{1}(n,k,k,1,\tau_{1}(\mf), \tau_{1}(\mg))$. It is sufficient to show $f_{1}(n,k,k,1,\tau_{1}(\mf), \tau_{t}(\mg))<h_{1}(n,k,k,1)$.

Suppose  $\tau_{1}(\mf)=\tau_{1}(\mg)=2$. This together with Lemma \ref{2605291} (i) yields $f_{1}(n,k,k,1,2, 2)<h_{1}(n,k,k,1)$.  Suppose  $\min\{\tau_{1}(\mf), \tau_{1}(\mg)\}\geq 3$. By Lemma \ref{2604281}, we know $$f_{1}(n,k,k,1,\tau_{1}(\mf),\tau_{1}(\mg))\leq f_{1}(n,k,k,1,3,k).$$ It follows from Lemma \ref{2605291} (i) that    $f_{1}(n,k,k,1,\tau_{1}(\mf),\tau_{1}(\mg))<h_{1}(n,k,k,1)$, as required.

(ii)  By $\tau_{1}(\mf)\geq 3$ and  $ \tau_{1}(\mg)\geq 2$,  Corollary \ref{2605191} implies $\left|\mf\right|+\left|\mg\right|\leq f_{2}(n,k,\ell,1,3,2)$. From Lemma \ref{2605291} (ii), we obtain  the desired result.
\end{proof}

\begin{lem}\label{2606064}
	Let notations be as in Assumption \ref{2606072}. 
	If $\tau_{1}(\mf)=2$ and $\mf$ is not $1$-intersecting,   then $\left|\mf\right|+\left|\mg\right|\leq h_{1}(n,k,\ell,1)$. 
	Moreover, equality holds if and only if $(\mf, \mg)$ is isomorphic to $(\mf_{1}, \mg_{1})$.
\end{lem}
\begin{proof}
	Pick $A, B\in\mf$ with $\dim(A\cap B)=0$. Then $\mt_{1}(\mf)\subseteq \mw_{1}(A,B)$. Suppose $\left|\mf\right|=2$. Then $\mf=\{A, B\}$ and $\mg\subseteq \mh(A,B;\ell ,t)$, which implies $\left|\mf\right|+\left|\mg\right|\leq h_{1}(n,k,\ell,1)$. Moreover, equality holds if and only if $\mg= \mh(A,B;\ell ,t)$, i.e., $(\mf, \mg)$ is isomorphic to $(\mf_{1}, \mg_{1})$.
	
	Suppose $\left|\mf\right|\geq 3$. It is sufficient to show $\left|\mf\right|+\left|\mg\right|<h_{1}(n,k,\ell,1)$. By Lemma \ref{2604212}, we get 
	\begin{equation*}
		\left|\mt_{t}(\mf)\right|\leq {k-1\brack 1}\left({k\brack 1}+q\right)=:z.
	\end{equation*} We divide our proof into the following cases.  
	
	\medskip
	\noindent \textbf{Case 1.} $\ell=k$.
	\medskip
	
	If $\tau_{1}(\mg)=2$, then by Lemma \ref{26204271} (i), we know the desired result holds. Next, we assume $\tau_{1}(\mg)\geq 3$.  Corollary \ref{2605192} thereby implies	$\left|\mf\right|+\left|\mg\right|\leq f_{3}\left(n,k,k,1,2,3, z\right)$. 
	This together with  Lemma \ref{260501} yields  $\left|\mf\right|+\left|\mg\right|<h_{1}(n,k,k,1)$.
	
	\medskip
	\noindent \textbf{Case 2.} $\ell>k$.
	\medskip

	Note that $\tau_{1}(\mg)\geq 2$.  By Corollary \ref{2605192}, we have  $\left|\mf\right|+\left|\mg\right|\leq f_{3}\left(n,k,\ell,1,2,2, z\right)$.  It follows from Lemmas \ref{260501} and \ref{2605295} that $\left|\mf\right|+\left|\mg\right|<h_{1}(n,k,\ell,1)$.	 This finishes the proof.
\end{proof}

\begin{lem}\label{2606062}
		Let notations be as in Assumption \ref{2606072}. 
		If $\tau_{1}(\mf)=2$ and $\mf$ is $1$-intersecting,  then $\left|\mf\right|+\left|\mg\right|< h_{1}(n,k,\ell,1)$. 
 \end{lem}
\begin{proof}
	Suppose that $\mf$ is $(k-1)$-intersecting. Note that $\tau_{1}(\mf)=2$. Then $\dim(\cap_{F\in\mf}F)=0$. By Proposition \ref{2605201}, there exists $C\in{V\brack k+1}$ such that $\mf\subseteq{C\brack k}$ and $\mg\subseteq \mm(C;\ell,1)$.
	Hence $\left|\mf\right|+\left|\mg\right|\leq m_{1}(n,k,\ell,1)$. This together with Lemma \ref{2604301} yields $\left|\mf\right|+\left|\mg\right|<h_{1}(n,k,\ell, 1)$.

	Next, we  assume $1\leq \dim(A\cap B)\leq k-2$ for some $A, B\in\mf$.  This implies $k\geq 3$.  	In the remaining of this proof, write $s=\dim(A\cap B)$ and 
	$$\ma=\left\{T\in\mt_{1}(\mf): \dim(T\cap A\cap B)\geq 1\right\},\ 
	\mb=\left\{T\in\mt_{1}(\mf): \dim(T\cap A\cap B)=0\right\}.$$
	Note that $\ma=\bigcup_{H\in {A\cap B\brack 1}}\ma_{H}$. Pick $H_{0}\in {A\cap B\brack 1}$ such that $\left|\ma_{H}\right|\leq \left|\ma_{H_{0}}\right|$ fro any $H\in{A\cap B\brack 1}$.  By $\tau_{1}(\mf)=2$, there exists $F_{0}\in \mf$ such that $\dim(H_{0}\cap F_{0})=0$. It follows from Lemma \ref{2604155} that $\left|\ma_{H_{0}}\right|\leq {k\brack 1}\left|\ma_{U}\right|\leq {k\brack 1}$ for some $2$-dimensional subspace $U$ of $V$. Hence
	$\left|\ma\right|\leq {s\brack 1}{k\brack 1}$.
	One can check  $\mb\subseteq\mw_{2}(A,B)$. This together with $\dim(\cap_{F\in\mf}F)=0$ and Lemma \ref{2604241} yields $\left|\mb\right|\leq q^{2s-1}{k-s\brack 1}\left({k-s\brack 1}+q\right)$. 
	From $\mt_{t}(\mf)=\ma\cup\mb$, we obtain
	\begin{equation*}
		\left|\mt_{1}(\mf)\right|\leq {s\brack 1}{k\brack 1}+q^{2s-1}{k-s\brack 1}\left({k-s\brack 1}+q\right)=:z.
	\end{equation*} 
	We divide our proof into the following cases.

	\medskip
	\noindent \textbf{Case 1.} $\ell=k$.
	\medskip
	
	If $\tau_{1}(\mg)=2$, then by Lemma \ref{26204271} (i), we have $\left|\mf\right|+\left|\mg\right|<h_{1}(n,k,k,1)$. The required result holds. Suppose $\tau_{1}(\mg)\geq 3$. It follows from Corollary \ref{2605192} that 
	$\left|\mf\right|+\left|\mg\right|\leq f_{3}(n,k,k,1,2,3, z)$.  By  Lemma \ref{260501}, we have $\left|\mf\right|+\left|\mg\right|<h_{1}(n,k,k,1)$. 
 
	\medskip
	\noindent \textbf{Case 2.} $\ell>k$.
	\medskip
	
	By $\tau_{1}(\mg)\geq 2$, Corollary \ref{2605192} yields 	$\left|\mf\right|+\left|\mg\right|\leq f_{3}(n,k,\ell,1,2,2, z)$. From   Lemma \ref{260501}, we obtain $\left|\mf\right|+\left|\mg\right|<h_{1}(n,k,\ell,1)$.  This finishes the proof.	
\end{proof}

\begin{proof}[\textnormal{\textbf{Proof of Theorem \ref{1}} (i) and (ii)}]
	
	(i) By Lemma \ref{26204271} (i) and (\ref{2606061}), we know $\tau_{1}(\mf)=2$   or  $\tau_{1}(\mg)=2$. 
	Suppose $\tau_{1}(\mf)=2$.  It follows from (\ref{2606061}) and Lemma \ref{2606062} that $\mf$ is not $1$-intersecting. Moreover, by Lemma \ref{2606064} and (\ref{2606061}), we have $(\mf,\mg)$ is isomorphic  to $(\mf_{1},\mg_{1})$. 	
	Suppose $\tau_{1}(\mg)=2$. By symmetry, we have $(\mg,\mf)$ is isomorphic  to $(\mf_{1},\mg_{1})$.

	(ii)  By Lemma \ref{26204271} (ii) and (\ref{2606061}), we know $\tau_{1}(\mf)=2$. By Lemma \ref{2606062}, we get $\mf$ is not $1$-intersecting.  It follows from Lemma \ref{2606064} and (\ref{2606061}) that $(\mf,\mg)$ is isomorphic  to $(\mf_{1},\mg_{1})$.
\end{proof}

\subsection{Proof of Theorem \ref{1} (iii)} \label{2606076}

Before presenting the proof of Theorem \ref{1} (iii), we establish the following auxiliary result.

\begin{lem}\label{2604131}
	Let  $n\geq t+3$. Suppose that $\mf,\mg\subseteq {V\brack t+1}$ are cross-$t$-intersecting.   If  $\mg$ is not $t$-intersecting and $\left|\mathcal{G}\right|\geq 3$, then $\left|\mathcal{F}\right|\leq 2q+1$. 
\end{lem}
\begin{proof}
	Pick $G_{1}, G_{2}\in\mg$ with $\dim(G_{1}\cap G_{2})<t$. Write $C=G_{1}\cap G_{2}$ and $S=G_{1}+G_{2}$. Since $\mf$ and $\mg$ are cross-$t$-intersecting, we have $\dim(C)=t-1$ and $C\subseteq F\subseteq S$ for any $F\in\mf$.

	\medskip
	\noindent \textbf{Case 1.} $C\subseteq G$ for any $G\in\mg$.
	\medskip

	Consider the following map
	$$\sigma: \left\{H\in {V\brack t+1}: C\subseteq H\right\}\longrightarrow {V/C\brack 2},\ H\longmapsto H/C.$$
	Note that $\sigma$ is a bijection. Then
	$\left|\mf\right|=\left|\left\{\sigma(F): F\in \mf\right\}\right|$ and $\left|\mg\right|=\left|\left\{\sigma(G): G\in \mg\right\}\right|$.   Since $\mf$ and $\mg$ are also cross-$t$-intersecting,  families $\left\{\sigma(F): F\in \mf\right\}$ and $\left\{\sigma(G): G\in \mg\right\}$ are cross-$1$-intersecting.  Moreover, we  know $\left\{\sigma(G): G\in \mg\right\}$ is not $1$-intersecting due to $G_{1}, G_{2}\in \mathcal{G}$. 
	 Therefore, it suffices to prove this  case for $t=1$. 	 
	  Indeed, since $\mf\subseteq \mw_{1}(G_{1},G_{2})$ and all members in $\mg$ are $2$-dimensional $1$-covers of $\mf$, by Lemma \ref{2604212}, we have $\left|\mf\right|\leq 2q+1$.

	\medskip
	\noindent \textbf{Case 2.} $C\nsubseteq G_{0}$ for some $G_{0}\in\mg$.
	\medskip

	For each $F\in\mf$, by $C\subseteq F$ and $\dim(F\cap G_{0})\geq t$, we have 
	$$t+1\geq \dim(C+(F\cap G_{0}))=\dim(F\cap G_{0})-\dim(C\cap G_{0})+t-1\geq t+1,$$
	which implies $F=C+(F\cap G_{0})$, $\dim(F\cap G_{0})=t$ and $\dim(C\cap G_{0})=t-2$. Hence
	\begin{equation}\label{2606071}
		\mf\subseteq \left\{C+T:T\in {G_{0}\cap S\brack t},\ C\cap G_{0}\subseteq T\right\}.
	\end{equation}
	Suppose $\dim(G_{0}\cap S)\leq t$. By $\mf\neq \emptyset$  and (\ref{2606071}), we have $\dim(G_{0}\cap S)=t$ and $\left|\mf\right|=1$, as required. Suppose $G_{0}\subseteq S$. Then $\min\{\dim(G_{0}\cap G_{1}), \dim(G_{0}\cap G_{2})\}\geq t-1$. We claim that $\min\{\dim(G_{0}\cap G_{1}), \dim(G_{0}\cap G_{1})\}=t-1$. Otherwise, $\dim((G_{0}\cap G_{1})+(G_{0}\cap G_{2}))\geq t+2$, a contradiction. W.l.o.g., we assume $\dim(G_{0}\cap G_{1})=t-1$.	
	Pick $T\in {G_{0}\brack t}$ with $C\cap G_{0}\subseteq T$. Then $C\cap G_{0}=C\cap T$.	
	If $C+T\in\mf$, then by $C+(T\cap G_{1})=(C+T)\cap G_{1}$, we have
	\begin{equation*}
		\begin{aligned}
			t\leq \dim(C+(T\cap G_{1}))
			=\dim(T\cap G_{1})+1\leq \dim(G_{0}\cap G_{1})+1=t,
		\end{aligned}
	\end{equation*}
 which implies $T\cap G_{1}=G_{0}\cap G_{1}$. It follows from (\ref{2606071})  that
 $$\mf\subseteq \left\{C+T:T\in {G_{0}\brack t},\ G_{0}\cap G_{1}\subseteq T\right\}.$$
 We further conclude $\left|\mf\right|\leq q+1$. The desired result holds.
\end{proof}

\begin{proof}[\textnormal{\textbf{Proof of Theorem \ref{1}} (iii)}]

	Suppose for contradiction   neither $\mf$ nor $\mg$ is $t$-intersecting. This implies $n\geq t+3$.  If $\left|\mf\right|\geq 3$ and $\left|\mg\right|\geq 3$, then by Lemma \ref{2604131}, we have $\left|\mf\right|+\left|\mg\right|\leq 4q+2$. Note that $4q+2<2{t+2\brack 1}=m_{1}(n,t+1,t+1,t)$. This  contradicts to (\ref{2606079}). If $\mf=\{A, B\}$, where $\dim(A\cap B)<t$, then by $\mg$ is non-empty, we get $\dim(A\cap B)=t-1$. Then  
	$$\left| \mg\right|\leq \left| \left\{A_{1}+B_{1}: A_{1}\in{A\brack t},\ A\cap B\subseteq A_{1},\ B_{1}\in{B\brack t},\ A\cap B\subseteq  B_{1}\right\}\right|\leq {2\brack 1}^{2}.$$
	It follows that $\left|\mf\right| +\left|\mg\right|\leq q^{2}+2q+3<m_{1}(n,t+1,t+1,t)$, a contradiction to (\ref{2606079}). If $\mg=\{C,D\}$, where $\dim(C\cap D)<t$, then by symmetry, we derive $\left|\mf\right| +\left|\mg\right|<m_{1}(n,t+1,t+1,t)$, a contradiction to (\ref{2606079}). Hence one of $\mf$ and $\mg$ is $t$-intersecting. 
	
	W.l.o.g., we may assume that $\mf$ is $t$-intersecting.  By Proposition \ref{2605201} and (\ref{2606079}),  there exists $C\in{ V\brack t+2}$ such that $\mf=\mg= {C\brack t+1}$, i.e., $(\mf,\mg)$ is isomorphic to $(\mf_{2},\mg_{2})$.
\end{proof}

\subsection{Proof of Theorem \ref{1} (iv) and (v)} \label{2606077}

The idea of proving Theorem \ref{1} (iv) and (v) is similar to that of (i) and (ii). To simplify the exposition, the following assumptions are introduced.

\begin{as}\label{2606078}
	Let $\min\{k,\ell\}\geq t+1\geq 3$ and $(k,\ell)\neq (t+1,t+1)$.  Assume that $n\geq k+\ell+1$ when $\min\{k,\ell\}=t+1$, and  $n\geq k+\ell-t+1+3\delta_{2,q}+\delta_{3,q}+\delta_{2,q}\delta_{2,t}$ when $\min\{k,\ell\}\geq  t+2$. Suppose that $\mf\subseteq{V\brack k}$ and $\mg\subseteq{V\brack \ell}$ are  non-trivial cross-$t$-intersecting. 
\end{as}

\begin{lem}\label{2605052}
	Let notations be as in Assumption \ref{2606078}. The following hold.
	\begin{itemize}
		\item[\normalfont(i)] If $\ell>k=t+1$ and $\tau_{t}(\mf)\geq t+2$, then $\left|\mf\right|+\left|\mg\right|< m_{1}(n,t+1,\ell,t)$.
		\item[\normalfont(ii)] If $\min\{k,\ell\}\geq  t+2$, and $\tau_{t}(\mf)=\tau_{t}(\mg)=t+1$ or  $\min\{\tau_{t}(\mf), \tau_{t}(\mg)\}\geq t+2$, then $\left|\mf\right|+\left|\mg\right|< \max\{m_{1}(n,k,\ell,t),m_{1}(n,\ell,k,t)\}$.
	\end{itemize}
\end{lem}
\begin{proof} (i) Since $\mf\subseteq{V\brack t+1}$ and $\mg\subseteq{V\brack \ell}$  are   non-trivial cross-$t$-intersecting, we have $\tau_{t}(\mg)=t+1$. It follows from Corollary \ref{2605191} that 
	$\left|\mf\right|+\left|\mg\right|\leq f_{2}(n,t+1,\ell,t,t+2,t+1).$
	We thereby derive the required result from Lemma \ref{2605294}.

(ii)	Suppose $\tau_{t}(\mf)\geq t+2$ and $\tau_{t}(\mg)\geq t+2$.  From Corollary \ref{2605191}, we obtain 
	$\left|\mf\right|+\left|\mg\right|\leq f_{2}(n,k,\ell,t,t+2,t+2)$. It follows from Lemma \ref{2605293} that the required result holds.

	Suppose $\tau_{t}(\mf)=\tau_{t}(\mg)=t+1$. By Corollary  \ref{2605191}, we have 	
	$\left|\mf\right|+\left|\mg\right|\leq f_{1}(n,k,\ell,t,t+1,t+1)$. This combining with Lemma \ref{2605293}  implies the desired result.
\end{proof}

\begin{lem}\label{2605056}
		Let notations be as in Assumption \ref{2606078}.  Suppose that $\tau_{t}(\mf)= t+1$ and $\mf$ is not $(k-1)$-intersecting.  The following hold.
		\begin{itemize}
			\item[\normalfont(i)] If $\ell>k=t+1$, then $\left|\mf\right|+\left|\mg\right|< m_{1}(n,t+1,\ell,t)$.
			\item[\normalfont(ii)] If $\min\{k,\ell\}\geq  t+2$, then $\left|\mf\right|+\left|\mg\right|< \max\{m_{1}(n,k,\ell,t),m_{1}(n,\ell,k,t)\}$.
		\end{itemize}
\end{lem}
\begin{proof}
	Pick $A, B\in\mf$ with $\dim(A\cap B)\leq k-2$. By $\tau_{t}(\mf)=t+1$, we know $\dim(A\cap B)\geq t-1$.  
	
	(i) Since $k=t+1$, we know  $\dim(A\cap B)=t-1$ and $\tau_{t}(\mg)=t+1$. Therefore, we have $$\mt_{t}(\mf)\subseteq \left\{A_{1}+B_{1}: A_{1}\in {A\brack t},\ A\cap B\subseteq A_{1},\ B_{1}\in {B\brack t},\ A\cap B\subseteq B_{1} \right\}.$$
	This implies $\left|\mt_{t}(\mf)\right|\leq {2\brack 1}^{2}$. From Corollary  \ref{2605192}, we obtain 		
	$$\left|\mf\right|+\left|\mg\right|\leq f_{3}(n,t+1,\ell,t,t+1,t+1, (q+1)^{2}).$$
	It follows from Lemma \ref{2605296} that $\left|\mf\right|+\left|\mg\right|<m_{1}(n,t+1,\ell, t)$.

	(ii) By Lemma \ref{2605052} (ii), the desired result holds if $\tau_{t}(\mg)=t+1$. Now, we may assume $\tau_{t}(\mg)\geq t+2$.  In the remaining of this proof, write $s=\dim(A\cap B)$,
	$$\ma=\left\{T\in\mt_{t}(\mf): \dim(T\cap A\cap B)\geq t\right\}, \ \mb=\left\{T\in\mt_{t}(\mf): \dim(T\cap A\cap B)=t-1\right\}.$$
	 
	 We proceed by showing $\left|\ma\right|\leq {s\brack t}{k-t+1\brack 1}$. If $\ma$ is empty, then it is clear. Suppose $\ma$ is non-empty. Then $s\geq t$ and  $\ma=\bigcup_{H\in {A\cap B\brack t}}\ma_{H}$. Choose $H_{0}\in {A\cap B\brack t}$ such that $\left|\ma_{H}\right|\leq \left|\ma_{H_{0}}\right|$ for any $H\in {A\cap B\brack t}$. By $\tau_{t}(\mf)=t+1$, we know $\dim(H_{0}\cap F_{0})<t$ for some $F_{0}\in\mf$. It follows from Lemma \ref{2604155} that $\left|\ma_{H_{0}}\right|\leq {k-t+1\brack 1}^{\dim(U)-\dim(H_{0})}\left|\ma_{U}\right|$ for some subspace $U$ of $V$ with $\dim(U)>t$. Note that $\ma_{U}\neq \emptyset$ due to $\ma\neq\emptyset$. This implies $\dim(U)=t+1$ and  $\left|\ma_{H_{0}}\right|\leq {k-t+1\brack 1}$. Hence $\left|\ma\right|\leq {s\brack t}{k-t+1\brack 1}$.
	
	Next, we prove $\left|\mb\right|\leq q^{2(s-t+1)}{s\brack t-1}{k-s\brack 1}^{2}$. For each $T\in {V\brack t+1}$ with $\dim(T\cap A)\geq t$, $\dim(T\cap B)\geq t$ and $\dim(T\cap A\cap B)=t-1$, we have 
	$$t+1\geq \dim((T\cap A)+(T\cap B))=\dim(T\cap A)+\dim(T\cap B)-t+1\geq t+1,$$ 
	which implies $T=(T\cap A)+(T\cap B)$, $\dim(T\cap A)=t$ and $\dim(T\cap B)=t$.  Hence
	\begin{equation*}
		\begin{aligned}
			\mb&\subseteq \left\{A_{1}+B_{1}: A_{1}\in {A\brack t},\ B_{1}\in{B\brack t},\ A_{1}\cap B=B_{1}\cap A\in{A\cap B\brack t-1}\right\}\\
			&=\bigcup_{W\in {A\cap B\brack t-1}} \left\{A_{1}+B_{1}: A_{1}\in {A\brack t},\ B_{1}\in{B\brack t},\ A_{1}\cap B=B_{1}\cap A=W\right\}.
		\end{aligned}
	\end{equation*}
	This yields  $\left|\mb\right|\leq q^{2(s-t+1)}{s\brack t-1}{k-s\brack 1}^{2}$.

	For each $T\in\mt_{t}(\mf)$, by $\dim(T\cap A)\geq t$ and $\dim(T\cap B)\geq t$, we have $\dim (T\cap A\cap B)\geq t-1$. 	Then $\mt_{t}(\mf)=\ma\cup\mb$.  It follows from $\tau_{t}(\mg)\geq t+2$ and  Corollary \ref{2605192} that $\left|\mf\right|+\left|\mg\right|\leq f_{3}(n,k,\ell,t,t+1, t+2,z)$,  
	where $z={s\brack t}{k-t+1\brack 1}+q^{2(s-t+1)}{s\brack t-1}{k-s\brack 1}^{2}$.
	 This together with  Lemma  \ref{2605071} implies $	\left|\mf\right|+\left|\mg\right|< \max\{m_{1}(n,k,\ell,t), m_{1}(n,\ell,k,t)\}$, as desired. 
\end{proof}

\begin{proof}[\textnormal{\textbf{Proof of Theorem \ref{1} (iv) and (v)}}]
	(iv) By (\ref{2606079}) and Lemma \ref{2605052} (i), we have $\tau_{t}(\mf)=t+1$. Lemma \ref{2605056} (i) implies  $\mf$ is $t$-intersecting. It follows from Proposition \ref{2605201} and (\ref{2606079}) that $\mf={C\brack t+1}$ and $\mg=\mm(C; \ell,t)$ for some $C\in{V\brack t+2}$, i.e., $(\mf,\mg)$ is isomorphic to $(\mf_{2},\mg_{2})$.

	(v)  
	From  (\ref{2606079}) and Lemma \ref{2605052} (ii), 	we have  $\tau_{t}(\mf)=t+1$ or $\tau_{t}(\mg)=t+1$.  Suppose  $\tau_{t}(\mf)=t+1$.  By (\ref{2606079}) and Lemma \ref{2605056} (ii), we know $\mf$ is $(k-1)$-intersecting. Proposition \ref{2605201} therefore implies that  $\mf\subseteq {C\brack k}$  and $\mg\subseteq\mm(C;\ell,t)$ for some $C\in{V\brack k+1}$. Combining this with (\ref{2606079}), we obtain
	$\mf={C\brack k}$ and $\mg=\mm(C;\ell,t)$, i.e., $(\mf,\mg)$ is isomorphic to $(\mf_{2},\mg_{2})$. 
	Suppose $\tau_{t}(\mg)=t+1$. Interchanging the roles of $(\mf,k)$ and $(\mg,\ell)$,  the same argument yields   $\mg={D\brack \ell}$  and  $\mf=\mm(D;k,t)$ for some $D\in{V\brack \ell+1}$, i.e., $(\mf,\mg)$ is isomorphic to $(\mf_{3},\mg_{3})$. 
\end{proof}

\section{Concluding remarks}\label{2606019}

Let $\min\{k,\ell\}\geq t+1$ and $n\geq k+\ell-t+1$.  Suppose that  $\mf\subseteq {V\brack k}$ and $\mg\subseteq {V\brack \ell}$ are cross-$t$-intersecting families. Theorem \ref{1} has  characterized the extremal families  when $\left|\mf\right|+\left|\mg\right|$ is maximum, subject to  $\dim(\cap_{F\in\mf}F)<t$ and $\dim(\cap_{G\in\mg}G)<t$.  In this section,  we  investigate the  problem of maximizing $|\mf|+|\mg|$ under a weaker condition  $\dim(\cap_{F\in\mf}F)<t$.  
The next result shows that, if $\ell\geq k+2$, then the  weaker hypothesis yields no new extremal families, just as in Theorem \ref{1}.

\begin{pr1}\label{2606091}
Let $\ell\geq k+2\geq t+3$. Suppose that $\mf\subseteq{V\brack k}$ and $\mg\subseteq{V\brack \ell}$ are cross-$t$-intersecting families with $\dim(\cap _{F\in\mf}F)<t$, and $\left|\mf\right|+\left|\mg\right|$ is maximum. 
\begin{itemize}
	\item[\normalfont(i)] If $t=1$, $n\geq k+\ell+1+2\delta_{2,q}$ and $(n,k,\ell)\neq (7,2,4)$, then  $(\mf,\mg)$ is isomorphic to $(\mf_{1},\mg_{1})$.
	\item[\normalfont(ii)] If $t\geq 2$, $n\geq k+\ell+\lceil\frac{t+1}{\ell-k-1}\rceil$ and $(n,t,q)\notin\{(k+\ell+1,2,2), (k+\ell+2,2,2)\}$, then $(\mf,\mg)$ is isomorphic to $(\mf_{2},\mg_{2})$.
\end{itemize}
\end{pr1}
\begin{proof}
Suppose to the contrary $\dim(\cap_{G\in\mg}G)\geq t$. By Corollary \ref{2605191} and Lemma \ref{2604157}, we have
$$
\left|\mathcal{F}\right|+\left|\mathcal{G}\right|
\leq f_{1}(n,k,\ell,t,\tau_{t}(\mathcal{F}),t)
\leq {\ell\brack t}{n-t\brack k-t}+{k-t+1\brack 1}{n-t-1\brack \ell-t-1}.$$
It follows from Lemmas \ref{2605311} and \ref{2605312} (i) that
$
\left|\mathcal{F}\right|+\left|\mathcal{G}\right|<h_{1}(n,k,\ell,1)
$
if $t=1$, and
$\left|\mathcal{F}\right|+\left|\mathcal{G}\right|<m_{1}(n,k,\ell,t)$
if $t\geq 2$. This contradicts the maximality of $\left|\mathcal{F}\right|+\left|\mathcal{G}\right|$. Hence $\mathcal{F}$ and $\mathcal{G}$ are non-trivial cross-$t$-intersecting families with the maximum sum of sizes. The desired assertion now follows from Theorem \ref{1} and Lemma \ref{2605312} (ii).
\end{proof}

Suppose that $\ell\leq k+1$. Then  there exist families whose sum of sizes exceeds that of $(\mf_{1}, \mg_{1})$ and $(\mf_{2},\mg_{2})$.  If $\ell=k$, or $\ell<k$ and $n\geq k+\ell$, then the extremal families were provided by \cite[Theorem 1.4]{WZ-2013-129}. 
 Suppose $\ell=k+1$. Pick $G\in{V\brack k+1}$. Consider the families $\left\{F\in{V\brack k}: \dim(F\cap G)\geq t\right\}$ and $\{G\}$. They are  cross-$t$-intersecting, and no $t$-subspace is contained in all members of the first family. Write the sum of their sizes as $g(n,k,t)$. Then $g(n,k,t)\geq \left|\left\{F\in{V\brack k}: \dim(F\cap G)= t\right\}\right|$, i.e.,
$$g(n,k,t)\geq q^{(k-t)(k-t+1)}{k+1\brack t}{n-k-1\brack k-t}\sim q^{(k-t)(n-k+t)+t}\ \ \text{as}\ \ q\to \infty.$$
By the definition of $h_{1}(n,k,k+1,1)$ and $m_{1}(n,k,k+1,t)$, we have 
$$h_{1}(n,k,k+1,1)\leq 2+{k\brack 1}^{2}{n-2\brack k-1}\sim q^{(k-1)(n-k+1)} \ \  \text{as}\ \ q\to \infty, $$
$$m_{1}(n,k,k+1,t)\leq {k+1\brack 1}+{k+1\brack t+1}{n-t-1\brack k-t}\sim q^{(k-t)(n-k+t)} \ \  \text{as}\ \  q\to \infty.$$
Hence $g(n,k,1)>h_{1}(n,k,k+1,1)$ and   $g(n,k,t)>m_{1}(n,k,k+1,t)$  if $q$ is large enough. 

\section{Some inequalities}\label{26060110}

This section is devoted to proving some inequalities concerning  Gaussian binomial 
coefficients. The first one can be easily verified.

\begin{lem}\label{2604282}
	Let  $n>k$. Then
	$q^{n-k}<\frac{q^{n}-1}{q^{k}-1}< \frac{q^{n-k+1}}{q-1}$ and $q^{k(n-k)}<{n\brack k}<\frac{q^{k(n-k+1)}}{(q-1)^{k}}$.
\end{lem}

\begin{lem}\textnormal{([\citen{WL-2026-106127}, Lemma 4.3 (i)])}\label{2604157}
	Let  $\min\{k,  \ell\}\geq t$ and $n\geq k+\ell-t+1$. The function ${\ell-t+1\brack 1}^{m-t}{n-m\brack k-m}$ is strictly decreasing as $m\in \{t,t+1,\ldots, k\}$ increases. 
\end{lem}

Recall that $f_{1}(n,k,\ell,t,x,y)$, $f_{2}(n,k,\ell,t,x,y)$, $f_{3}(n,k,\ell,t,x,y,z)$ are defined in (\ref{2604193}), (\ref{2605193}), (\ref{2605057}) respectively, and $h_{2}(n,k,\ell,t)$, $m_{2}(n,k,\ell,t)$ are defined in (\ref{2604174}), (\ref{2604175}) respectively.
 
\begin{lem}\label{2604281}
	Let $k\geq 3$ and $n\geq 2k+\delta_{2,q}$. Then
	$f_{1}(n,k,k,1,x,y )\leq f_{1}(n,k,k,1,3,k)$ for any $3\leq x,y\leq k$.
\end{lem}
\begin{proof}
    If $k=3$, then the assertion is clear. Next, we assume $k\geq 4$.

 Pick $3\leq w<k $.     Lemma \ref{2604282} gives  
 $\frac{q^{n-w}-1}{q^{k-w}-1}-\frac{q^{k}-1}{q-1}-q^{w} \geq  \frac{q^{n-k}(q-1)-2q^{k}+q^{k-1}+1}{q-1}\geq 0$. 
 From Lemma \ref{2604157}, we obtain ${k\brack 1}^{k-1}\leq {k\brack 1}^{w}{n-w-1\brack k-w-1}$ and ${k\brack 1}^{w-1}{n-w\brack k-w}\leq {k\brack 1}^{2}{n-3\brack k-3}$. Hence
    \begin{equation*}
    	\begin{aligned}
    		\frac{f_{1}(n,k,k,1,w,k)-f_{1}(n,k,k,1,w+1,k)}{{k\brack 1}^{w} {n-w-1\brack k-w-1}}
    		\geq  \frac{q^{n-w}-1}{q^{k-w}-1}-\frac{q^{k}-1}{q-1}-q^{w} \geq 0,
    	\end{aligned}
    \end{equation*}
	\begin{equation*}
			\begin{aligned}
			 f_{1}(n,k,k,1,3,w+1)-f_{1}(n,k,k,1,3,w)
				\geq{k\brack 1}^{2}{n-3\brack k-3}(q^{w}-q^{2}-q-1)\geq 0.
			\end{aligned}
	\end{equation*}
Therefore, we have 
\begin{equation}\label{2604283}
	\max\{f_{1}(n,k,k,1,w,k),\ f_{1}(n,k,k,1,3,w)\}\leq f_{1}(n,k,k,1,3,k) \textnormal{ for any } 3\leq w\leq k.
\end{equation}  	
This implies the desired result if $k=4$ due to the symmetry. Assume $k\geq 5$ in the following.

	Choose  $4\leq x\leq  y\leq k-1$. By Lemma \ref{2604157}, we conclude 
	\begin{equation*}
		\begin{aligned}
			&\ f_{1}(n,k,k,1,x-1,y+1)-f_{1}(n,k,k,1,x,y)\\
			\geq&{k\brack 1}^{x-2}{n-x\brack k-x}\left( \frac{q^{n-x+1}-1}{q^{k-x+1}-1}{y+1\brack 1}-{k\brack 1}{y\brack 1}\right)-{x\brack 1}{k\brack 1}^{y-1}{n-y\brack k-y}\\
			\geq &\ {k\brack 1}^{x-2}{n-x\brack k-x}\left(  \frac{q^{n-x+1}-1}{q^{k-x+1}-1}{y+1\brack 1}-{k\brack 1}{x\brack 1}-{k\brack 1}{y\brack 1}\right). 
		\end{aligned}
	\end{equation*}
Note that ${y+1\brack 1}\geq {x\brack1}+{y\brack 1}$ due to $y\geq x$. It follows from Lemma \ref{2604282} that 
$\frac{q^{n-x+1}-1}{q^{k-x+1}-1}{y+1\brack 1}\geq{x\brack 1}{k\brack 1}+{k\brack 1}{y\brack 1}$. Therefore, we have
\begin{equation}\label{2604284}
	f_{1}(n,k,k,1,x-1,y+1)\geq f_{1}(n,k,k,1,x,y)\textnormal{ for any } 4\leq x\leq  y\leq k-1.
\end{equation}

Let  $3\leq x,y\leq k$. Since $f_{1}(n,k,k,1, x,y)=f_{1}(n,k,k,1, y,x)$, we may assume $x\leq y$.  Suppose $x=3$ or $y=k$. By  (\ref{2604283}), the required result follows. Suppose $4\leq x\leq y\leq k-1$. If $x-3\leq k-y$, then by (\ref{2604284}) repeatedly, we know $f_{1}(n,k,k,1, x,y)\leq f_{1}(n,k,k,1,3, y+x-3)$. It follows from  (\ref{2604283}) that the desired result holds. If $x-3> k-y$, then by similar discussion, we get 
$f_{1}(n,k,k,1, x,y)\leq f_{1}(n,k,k,1,x+y-k, k)\leq f_{1}(n,k,k,1, 3,k)$, as desired. 
\end{proof}

 Note that $h_{1}(n,k,\ell,t)$ and $m_{1}(n,k,\ell,t)$ are the sums of sizes of families in Constructions \ref{2604171} and \ref{2604172}, respectively.

\begin{lem}\label{2605291}
	Let $\ell\geq k\geq 2$, $(k,\ell)\neq (2,2)$ and $n\geq k+\ell+1+\delta_{2,q}-\delta_{k, \ell}$. 
	\begin{itemize}
		\item[\normalfont(i)] If $k=\ell$, then $h_{1}(n,k,k,1)>\max\{f_{1}(n,k,k,1,2,2), f_{1}(n,k,k,1,3,k)\}$.
		\item[\normalfont(ii)] If $k<\ell$, then $h_{1}(n,k,\ell,1)>f_{2}(n,k,\ell,1,3,2).$
	\end{itemize}
\end{lem}
\begin{proof}
	
	(i)  We first show $f_{1}(n,k,k,1,2,2)<h_{1}(n,k,k,1)$. Suppose $(k,q)=(3,2)$. Then $f_{1}(n,3,3,1,2,2)=42(2^{n-2}-1)$ and $h_{2}(n,3,3,1)=49(2^{n-2}-5)$. Note that $42(2^{n-2}-1)<49(2^{n-2}-5)$ due to $n\geq 7$.  Suppose $(k,q)\neq (3,2)$. Then
	$$\frac{f_{1}(n,k,k,1,2,2)}{{k\brack 1}^{2}{n-2\brack k-2}}=\frac{2(q+1)}{{k\brack 1}}\leq \frac{2}{{k-1\brack 1}}+\frac{2}{{k\brack 1}}\leq \frac{2}{3}.$$ 
	By Lemma \ref{2604173},   the required result holds.
	
	Next, we prove $f_{1}(n,k,k,1,3,k)< h_{1}(n,k,k,1)$. By Lemma \ref{2604157},  we get 
	$f_{1}(n,k,k,1,3,k)\leq {3\brack 1}{k\brack 1}^{2}{n-3\brack k-3}+{k\brack 1}^{3}{n-3\brack k-3}$.
	It follows that
	\begin{equation*}
		\begin{aligned}
			&\ h_{2}(n,k,k,1)-f_{1}(n,k,k,1,3,k)\\
			\geq&\ {k\brack 1}^{2}{n-2\brack k-2}-2q{k\brack 2}{k\brack 1}{n-3\brack k-3}-{3\brack 1}{k\brack 1}^{2}{n-3\brack k-3}-{k\brack 1}^{3}{n-3\brack k-3}\\
			=&\ {k\brack 1}^{2}{n-3\brack k-3}\left( \frac{q^{n-2}-1}{q^{k-2}-1}-2q\frac{q^{k-1}-1}{q^{2}-1}-{3\brack1}-{k\brack 1}\right).
		\end{aligned}
	\end{equation*}
	Suppose $k=3$. From $n\geq 6+\delta_{2,q}$, we obtain
	$\frac{q^{n-2}-1}{q-1}-2q-2{3\brack 1}=\frac{q^{n-2}-2q^{3}-2q^{2}+2q+1}{q-1}\geq 0$.
	Suppose $k\geq 4$. By Lemma \ref{2604282}, we have 	$\frac{q^{n-2}-1}{q^{k-2}-1}\geq q^{n-k}$, which implies 
	\begin{equation*}
		\begin{aligned}
			\frac{q^{n-2}-1}{q^{k-2}-1}-2q\frac{q^{k-1}-1}{q^{2}-1}-{k\brack 1}-{3\brack1}\geq& \frac{q^{n-k}(q^{2}-1)-q^{k+1}-3q^{k}-q^{4}-q^{3}+4q+2}{q^{2}-1}.
		\end{aligned}
	\end{equation*}
	If $q=2$, then by $n\geq 2k+1$, we know 
	\begin{equation*}
		\begin{aligned}
			3\cdot2^{n-k}- 2^{k+1}-3\cdot 2^{k}-14  \geq 3\cdot 2^{k+1}- 2^{k+1}-3\cdot2^{k}-14=2^{k}-14\geq 0. 
		\end{aligned}
	\end{equation*}
	If $q\geq 3$, then by $n\geq 2k$ and $q^{k+2}\geq q^{k+1}+6q^{k}$, we get 
	\begin{equation*}
		\begin{aligned}
			q^{n-k}(q^{2}-1)-q^{k+1}-3q^{k}-q^{4}-q^{3}+4q+2
			\geq q^{k+2}-q^{k+1}-4q^{k}-q^{4}-q^{3}+4q+2 \geq 0. 
		\end{aligned}
	\end{equation*}
	In summary, inequality $f_{1}(n,k,k,1,3,k)\leq h_{2}(n,k,k,1)$ holds. This together with Lemma \ref{2604173} yields the desired result. 
	
	(ii) From Lemma \ref{2604282}, we obtain 
	\begin{equation*}
		\begin{aligned}
			\frac{f_{2}(n,k,\ell,1,3,2)}{{k\brack 1}^{2}{n-2\brack \ell-2}}&= \left(q^{\ell-k}+\frac{q^{\ell-k}-1}{q^{k}-1}\right)^{2}\cdot\prod_{i=1}^{\ell-k}\frac{q^{\ell-i-1}-1}{q^{n-k-i+1}-1}+\frac{{k\brack 1}(q^{\ell-2}-1)}{q^{n-2}-1}\\
			&\leq  \frac{16q^{2(\ell-k)}}{9}\cdot \frac{1}{q^{(\ell-k)(n-k-\ell+2)}} +\frac{1}{(q-1)q^{n-k-\ell}}\\
			&=\frac{16}{9q^{(\ell-k)(n-k-\ell)}}+\frac{1}{(q-1)q^{n-k-\ell}}.
		\end{aligned}
	\end{equation*}
	Note that $n\geq k+\ell+1+\delta_{2,q}$. Then $f_{2}(n,k,\ell,1,3,2)\leq \frac{5}{6}{k\brack 1}^{2}{n-2\brack \ell-2}$. This together with Lemma \ref{2604173} implies the desired result.
\end{proof}

\begin{lem}\label{2604301}
 Let $\ell\geq k\geq 2$, $(k,\ell)\neq (2,2)$ and $n\geq k+\ell+\delta_{2,q}$. Then	$h_{1}(n,k,\ell,1)>m_{1}(n,k,\ell,1)$.
\end{lem}
\begin{proof}
By Lemma \ref{2604173}, it is sufficient to show $h_{2}(n,k,\ell,1)\geq m_{1}(n,k,\ell,1)$. From the definition of $m_{1}(n,k,\ell,1)$, we obtain 
$m_{1}(n,k,\ell, 1)\leq {k+1\brack 1}+{k+1\brack 2}{n-2\brack \ell-2}$.  Note that $\ell\geq 3$. We further conclude 
\begin{equation*}
	\begin{aligned}
		h_{2}(n,k,\ell,1)-m_{1}(n,k,\ell,1)\geq q{k\brack2}{n-3\brack \ell-3}\left(\frac{q^{n-2}-1}{q^{\ell-2}-1}-2{k\brack 1}\right)-{k+1\brack 1}.
	\end{aligned}
\end{equation*}
If $\ell=3$, then $$\frac{q^{n-2}-1}{q-1}-2{k\brack 1}=\frac{q^{n-2}-2q^{k}+1}{q-1}\geq \frac{q^{k}(q^{1+\delta_{2,q}}-2)+1}{q-1}\geq \frac{q^{k}}{q-1}.$$
If $\ell\geq 4$, then by Lemma \ref{2604282}, we have ${n-3\brack \ell-3}\geq q^{(\ell-3)(n-\ell)}\geq q^{k}$ and 
$$\frac{q^{n-2}-1}{q^{\ell-2}-1}-2{k\brack 1}\geq q^{n-\ell}-2{k\brack 1}\geq \frac{q^{k+\delta_{2,q}}(q-1)-2q^{k}+2}{q-1}\geq \frac{1}{q-1},$$
which implies 
${n-3\brack \ell-3}\left(\frac{q^{n-2}-1}{q^{\ell-2}-1}-2{k\brack 1}\right)\geq \frac{q^{k}}{q-1}$. Hence 
$$	h_{2}(n,k,\ell,1)-m_{1}(n,k,\ell,1)\geq \frac{q^{k+1}}{q-1}{k\brack 2}-{k+1\brack 1}\geq 0.$$
 This implies the desired result.
\end{proof}

\begin{lem}\label{260501}
	Let $\ell\geq k\geq 3$ and $n\geq k+\ell+3\delta_{2,q}+\delta_{3,q}$. Then $$h_{1}(n,k,\ell,1)>f_{3}(n,k,\ell,1, 2, 2+\delta_{k,\ell},  z)$$ 
	for any $z\in\left\{{k-1\brack 1}\left({k\brack 1}+q\right)\right\}\cup \left\{{s\brack 1}{k\brack 1}+q^{2s-1}{k-s\brack 1}\left({k-s\brack 1}+q\right): 1\leq s\leq k-2 \right\}$.
\end{lem}
\begin{proof} 	 
	
  We first claim that, as   $s$ increases from $1$ to $k-2$, the function $q^{2s-1}{k-s\brack 1}\left({k-s\brack 1}+q\right)$ decreases  if $q=2$, and increases if $q\geq 3$.  Indeed, if $k=3$, then the required result is clear. Now we assume $k\geq 4$.    	Pick $1\leq s\leq k-3$.   Then 
	\begin{equation*}
		\begin{aligned}
			&\ q^{2s+1}{k-s-1\brack 1}\left({k-s-1\brack 1}+q\right)-q^{2s-1}{k-s\brack 1}\left({k-s\brack 1}+q\right)\\
			=&\ \frac{q^{2s-1}(q+1)}{q-1}\left( q^{k-s}(q-2)-q^{2}+q+1\right). 
		\end{aligned}
	\end{equation*}
	Note that $q^{k-s}(q-2)-q^{2}+q+1$ is negative if $q=2$, and is  positive if $q\geq 3$.  Consequently, $q^{2s-1}{k-s\brack 1}\left({k-s\brack 1}+q\right)$ attains its maximum value at $s=1$ if $q=2$, and at $s=k-2$ if $q \geq 3$. This implies  the claim.
	
   Moreover, we conclude that  ${k-2\brack 1}{k\brack1}+q^{2k-5}(q+1)(2q+1)$  minus ${k-1\brack 1}\left({k\brack 1}+q\right)$ is equal to $\frac{q^{2k-5}(q^{3}+q^{2}-2q-1)-q^{k}+q^{k-2}+1}{q-1}+1$, and this difference is greater than 0.

   	By the above discussion, for each $z$ in the assumption of this lemma, we have
	     $$z-{k-2\brack 1}{k\brack 1}\leq  \left\{
	     \begin{aligned}
	     	&2^{2k-1}-2 &\textnormal{if}\ \  q=2,\\
	     	&q^{2k-5}(q+1)(2q+1) &\textnormal{if}\ \  q\geq3.
	     \end{aligned}
	     \right. $$
	Note that $f_{3}(n,k,\ell, 1,2, 2+\delta_{k,\ell},  z+1)\geq f_{3}(n,k,\ell, 1,2, 2+\delta_{k,\ell},  z)$. We just need to show $h_{1}(n,k,\ell,1)> f_{3}(n,k,\ell, 1,2, 2+\delta_{k,\ell},  z)$ after setting $z$ to be the above upper bounds.

	\medskip
	\noindent \textbf{Case 1.} $q=2$.  
	\medskip

	By Lemma \ref{2604173}, we know $h_{1}(n,k,\ell,1)>\frac{11}{12}{k\brack 1}^{2}{n-2\brack \ell-2}$. It is sufficient to prove
	\begin{equation}\label{2605011}
		\begin{aligned}
			\frac{f_{3}(n,k,\ell,1, 2, 2+\delta_{k,\ell},  z)}{{k\brack 1}^{2}{n-2\brack \ell-2}}
			= \frac{z}{(2^{k}-1)^{2}}+\frac{{k\brack 1}{n-3\brack\ell-3}}{{n-2\brack \ell-2}}+\frac{3{\ell\brack1}^{1+\delta_{k,\ell}}{n-2-\delta_{k,\ell}\brack k-2-\delta_{k, \ell}}}{{k\brack 1}^{2}{n-2\brack\ell-2}}\leq \frac{11}{12},
		\end{aligned}
	\end{equation} 
where $z=(2^{k-2}-1)(2^{k}-1)+2^{2k-1}-2$. Note that
\begin{equation}\label{2605012}
	\frac{(2^{k-2}-1)(2^{k}-1)+2^{2k-1}-2}{(2^{k}-1)^{2}}=\frac{73}{96}-\frac{(2^{k}-13)^{2}}{96(2^{k}-1)^{2}}\leq \frac{73}{96}.
\end{equation} 
By Lemma \ref{2604282}, we have 
\begin{equation}\label{2605013}
	\frac{{k\brack 1}{n-3\brack \ell-3}}{{n-2\brack \ell-2}}=\frac{(2^{k}-1)(2^{\ell-2}-1)}{2^{n-2}-1}\leq \frac{2^{k}-1}{2^{n-\ell}}\leq \frac{1}{2^{n-k-\ell}}\leq \frac{1}{8}.
\end{equation}
If $k=\ell$, then by $\frac{2^{k-2}-1}{2^{2k+1}-1}\geq \frac{2^{k-1}-1}{2^{2k+3}-1}$ for any $k\geq 3$, we have
$$\frac{3{n-3\brack k-3}}{{n-2\brack k-2}}=\frac{3(2^{k-2}-1)}{2^{n-2}-1}\leq \frac{3(2^{k-2}-1)}{2^{2k+1}-1}\leq \frac{3}{127}. $$	
This together with (\ref{2605012}) and (\ref{2605013}) yields (\ref{2605011}). If $k<\ell$, then by Lemma \ref{2604282}, we get
\begin{equation*}
	\begin{aligned}
		 \frac{3{\ell\brack 1}{n-2\brack k-2}}{{k\brack 1}^{2}{n-2\brack \ell-2}}&= \frac{3(2^{\ell}-1)}{(2^{k}-1)^{2}}\cdot	\prod_{i=1}^{\ell-k}\frac{2^{\ell-i-1}-1}{2^{n-k-i+1}-1}\leq\frac{3(2^{\ell}-1)}{(2^{k}-1)^{2}}\cdot \frac{1}{2^{(\ell-k)(n-k-\ell+2)}}\\
		 &\leq \frac{3\cdot2^{\ell}}{(2^{k}-1)^{2}\cdot 2^{5(\ell-k)}}=\frac{3}{(2^{k}-2+2^{-k})\cdot 2^{4(\ell-k)}}\leq \frac{3}{98}.
	\end{aligned}
\end{equation*}
It follows from (\ref{2605012}) and (\ref{2605013}) that (\ref{2605011}) holds. The proof of this case is completed.

	\medskip
	\noindent \textbf{Case 2.} $q\geq 3$.  
	\medskip

	From Lemma \ref{2604173}, we obtain $h_{1}(n,k,\ell,1)>\frac{13}{15}{k\brack 1}^{2}{n-2\brack \ell-2}$. It is sufficient to prove
	\begin{equation}\label{2605021}
		\begin{aligned}
			\frac{f_{3}(n,k,\ell,1, 2, 2+\delta_{k,\ell},  z)}{{k\brack 1}^{2}{n-2\brack \ell-2}}
			= \frac{z}{{k\brack 1}^{2}}+\frac{{k\brack 1}{n-3\brack\ell-3}}{{n-2\brack \ell-2}}+\frac{(q+1){\ell\brack1}^{1+\delta_{k,\ell}}{n-2-\delta_{k,\ell}\brack k-2-\delta_{k, \ell}}}{{k\brack 1}^{2}{n-2\brack\ell-2}}\leq \frac{13}{15},
		\end{aligned}
	\end{equation} 
	where $z={k-2\brack 1}{k\brack 1}+q^{2k-5}(q+1)(2q+1)$. By Lemma \ref{2604282}, we have 
	\begin{equation}\label{2605022}
	\frac{{k\brack 1}{n-3\brack\ell-3}}{{n-2\brack \ell-2}}=\frac{(q^{k}-1)(q^{\ell-2}-1)}{(q-1)(q^{n-2}-1)}\leq \frac{1}{(q-1)q^{n-k-\ell}}\leq \frac{1}{(q-1)q^{\delta_{3,q}}}. 
	\end{equation}
	
	\medskip
	\noindent \textbf{Case 2.1.} $q=3$.  
	\medskip
	
Note that  ${k\brack 1}\geq q^{k-3}(q^{2}+q+1)$. Hence $28\cdot3^{2k-5}/{k\brack 1}^{2}\leq \frac{84}{169}$. It follows from Lemma \ref{2604282} that
	\begin{equation}\label{2605023}
		\frac{{k-2\brack 1}{k\brack 1}+28\cdot3^{2k-5}}{{k\brack 1}^{2}}\leq \frac{3^{k-2}-1}{3^{k}-1}+\frac{84}{169}\leq \frac{1}{9}+\frac{84}{169}=\frac{925}{1521}.
	\end{equation}
If $k=\ell$, then by Lemma \ref{2604282}, we know 
$\frac{4{n-3\brack k-3}}{{n-2\brack k-2}}\leq \frac{4}{3^{n-k}}\leq \frac{4}{81}$. This together with (\ref{2605022}) and (\ref{2605023}) yields (\ref{2605021}). If $k<\ell$, then  by Lemma \ref{2604282}, we get
\begin{equation*}
	\begin{aligned}
		\frac{4{\ell\brack1}{n-2\brack k-2}}{{k\brack 1}^{2}{n-2\brack\ell-2}}&=\frac{8(3^{\ell}-1)}{(3^{k}-1)^{2}}\cdot\prod_{i=1}^{\ell-k}\frac{3^{\ell-i-1}-1}{3^{n-k-i+1}-1}\leq \frac{8\cdot 3^{\ell}}{3^{2k}-2\cdot 3^{k}+1}\cdot \frac{1}{3^{(\ell-k)(n-k-\ell+2)}}\\
		&\leq \frac{8\cdot 3^{\ell}}{3^{2k}-2\cdot 3^{k}+1}\cdot \frac{1}{3^{3(\ell-k)}}=\frac{8}{3^{k}-2+3^{-k}}\cdot\frac{1}{3^{2(\ell-k)}}\leq \frac{6}{169}.
	\end{aligned}
\end{equation*}
	It follows from (\ref{2605022}) and (\ref{2605023}) that (\ref{2605021}) holds. The proof of this case is finished.

	\medskip
	\noindent \textbf{Case 2.2.} $q\geq 4$.  
	\medskip
	
	In this case, we have $\frac{2}{q}-\frac{3}{q^{3}}+\frac{1}{q^{4}}+\frac{1}{q^{5}}\leq \frac{469}{1024}$. It follows from Lemma \ref{2604282} that 
	\begin{equation}\label{2605024}
		\begin{aligned}
			&\ \frac{{k-2\brack 1}{k\brack 1}+q^{2k-5}(q+1)(2q+1)}{{k\brack 1}^{2}}\\
			=&\ \frac{2q+1}{q(q+1)}+\frac{q^{2}-3q-2}{q(q+1)}\cdot\frac{q^{k-2}-1}{q^{k}-1}+\frac{2q+1}{q(q+1)}\cdot\left(\frac{q^{k-2}-1}{q^{k}-1}\right)^{2}\\
			\leq &\ \frac{2q+1}{q(q+1)}+\frac{q^{2}-3q-2}{q^{3}(q+1)}+\frac{2q+1}{q^{5}(q+1)}=\frac{2}{q}-\frac{3}{q^{3}}+\frac{1}{q^{4}}+\frac{1}{q^{5}}\leq \frac{469}{1024}.
		\end{aligned}
	\end{equation}
	If $k=\ell$, then by $\frac{q^{k-2}-1}{q^{2k-2}-1}\geq \frac{q^{k-1}-1}{q^{2k}-1}$ for any $k\geq 3$, we have
	$$\frac{(q+1){n-3\brack k-3}}{{n-2\brack k-2}}=\frac{(q+1)(q^{k-2}-1)}{q^{n-2}-1}\leq  \frac{(q+1)(q^{k-2}-1)}{q^{2k-2}-1}\leq \frac{(q+1)(q-1)}{q^{4}-1}=\frac{1}{q^{2}+1}\leq \frac{1}{17}.$$
	It follows from (\ref{2605022}) and (\ref{2605024}) that (\ref{2605021}) holds. If $k<\ell$, then by Lemma \ref{2604282} and $\frac{q^{k}}{(q^{k}-1)^{2}}\geq \frac{q^{k+1}}{(q^{k+1}-1)^{2}}$ for any $k\geq 3$, we get
	\begin{equation*}
		\begin{aligned}
		\frac{(q+1){\ell\brack1}{n-2\brack k-2}}{{k\brack 1}^{2}{n-2\brack\ell-2}}&=\frac{(q^{2}-1)(q^{\ell}-1)}{(q^{k}-1)^{2}}\cdot\prod_{i=1}^{\ell-k}\frac{q^{\ell-i-1}-1}{q^{n-k-i+1}-1}\leq \frac{q^{\ell}(q^{2}-1)}{(q^{k}-1)^{2}}\cdot\frac{1}{q^{(\ell-k)(n-k-\ell+2)}}\\
		&\leq  \frac{q^{\ell}(q^{2}-1)}{(q^{k}-1)^{2}}\cdot\frac{1}{q^{2(\ell-k)}}=\frac{q^{k}}{(q^{k}-1)^{2}}\cdot\frac{(q^{2}-1)}{q^{\ell-k}}\leq \frac{q^{3}}{(q^{3}-1)^{2}}\cdot \frac{q^{2}-1}{q}\\
		&= \frac{q^{2}}{q^{3}-1}\cdot \frac{q^{2}-1}{q^{3}-1}\leq \frac{q}{q^{3}-1}\leq \frac{4}{63}. 
		\end{aligned}
	\end{equation*}
This together with (\ref{2605022}) and (\ref{2605024}) yields (\ref{2605021}). The proof of (ii) is completed.
\end{proof}

\begin{lem}\label{2605295}
	Let $\ell\geq 3$ and $n\geq \ell+3+\delta_{2,q}$. Then 
	$h_{1}(n,2,\ell,1)>f_{3}(n,2,\ell,1,2,2,2q+1)$.
\end{lem}
\begin{proof}
	By Lemma \ref{2604173},  we just need to show
	\begin{equation}\label{0605025}
		\begin{aligned}
			&\ h_{2}(n,2,\ell,1)-f_{3}(n,2,\ell,1, 2, 2,  2q+1)\\
			=&\ q^{2}{n-2\brack \ell-2}-(q+1)(q^{2}+4q+1){n-3\brack \ell-3}-(q+1){\ell\brack 1}\geq 0.
		\end{aligned}
	\end{equation}
 By Lemma \ref{2604282}, we have
	\begin{equation*}
		\begin{aligned}
			\frac{(q+1)(q^{2}+4q+1){n-3\brack \ell-3}}{q^{2}{n-2\brack \ell-2}}\leq \frac{(q+1)(q^{2}+4q+1)}{q^{5+\delta_{2,q}}}
			=\frac{1}{q^{2+\delta_{2,q}}}\left(1+\frac{5}{q}+\frac{5}{q^{2}}+\frac{1}{q^{3}}\right)\leq \frac{39}{64},
		\end{aligned}
	\end{equation*} 
	\begin{equation*}
		\begin{aligned}
			\frac{(q+1){\ell\brack 1}}{q^{2}{n-2\brack \ell-2}}\leq \frac{(q+1)q^{\ell}}{(q-1)q^{(\ell-2)(3+\delta_{2,q})+2}}=\left(1+\frac{2}{q-1}\right)\frac{1}{q^{(2+\delta_{2,q})(\ell-2)}}\leq \frac{3}{8}. 
		\end{aligned}
	\end{equation*}   
	Then (\ref{0605025}) holds. This finishes the proof.
\end{proof}

\begin{lem}\label{2605311}
	Let $\ell\geq k+2\geq 4$, $n\geq k+\ell+1+2\delta_{2,q}$ and $(n,k,\ell)\neq (7,2,4)$. Then $$h_{1}(n,k,\ell,1)> {\ell\brack 1}{n-1\brack k-1}+{k\brack 1}{n-2\brack \ell-2}.$$
\end{lem}
\begin{proof}
	
By Lemma \ref{2604173}, it is sufficient to show $h_{2}(n,k,\ell,1)- {\ell\brack 1}{n-1\brack k-1}-{k\brack 1}{n-2\brack \ell-2}\geq 0$, i.e., 
\begin{equation*}
	\begin{aligned}
		q{k-1\brack 1}{k\brack 1}{n-2\brack \ell-2}-2q{k\brack 2}{k\brack 1}{n-3\brack \ell-3}-{\ell\brack 1}{n-1\brack k-1}\geq 0.
	\end{aligned}
\end{equation*}
In the remaining of this proof, write
$$X=	\frac{2q{k\brack 2}{k\brack 1}{n-3\brack \ell-3}}{q{k-1\brack 1}{k\brack 1}{n-2\brack \ell-2}},\ \ Y=\frac{{\ell\brack 1}{n-1\brack k-1}}{q{k-1\brack 1}{k\brack 1}{n-2\brack \ell-2}}$$
 and $\Delta=2\delta_{2,q}+\delta_{2,k}\delta_{4,\ell}(1-\delta_{2,q})$. We just need to show $X+Y\leq 1$.  
 
 From Lemma \ref{2604282}, we obtain 
$X=\frac{2(q^{k}-1)(q^{\ell-2}-1)}{(q^{2}-1)(q^{n-2}-1)}\leq \frac{2}{q(q^{2}-1)}$ and
\begin{equation*}
	\begin{aligned}
		Y&=\frac{q-1}{q}\cdot\frac{q^{\ell}-1}{q^{n-k-1}-1}\cdot\left(1+\frac{1}{q^{k-1}-1}+\frac{1}{q^{n-k}-1}\right)\cdot \prod_{i=1}^{\ell-k-2}\frac{q^{\ell-i-1}-1}{q^{n-k-i-1}-1}\\
		&\leq \frac{q-1}{q}\cdot \frac{1}{q^{\Delta}}\cdot\left(1+\frac{1}{q^{k-1}-1}+\frac{1}{q^{\ell+1}-1} \right)\cdot \frac{1}{q^{\ell-k-2}}.
	\end{aligned}
\end{equation*}	
If $q=2$, then $X\leq \frac{1}{3}$ and $Y\leq \frac{63}{248}$.  If  $q\geq 3$ and $k=2$, then $X\leq \frac{1}{12}$ and  $Y\leq \frac{182}{363}$ due to $\ell+\Delta-4\geq 1$. If $q\geq 3$ and $k\geq 3$, then 
$$X+Y\leq\frac{2}{q(q^{2}-1)}+  \frac{q-1}{q}\cdot\left(1+\frac{1}{q^{2}-1}+\frac{1}{q^{6}-1} \right)=1+\frac{1}{q}\left(\frac{1}{q-1}+\frac{1}{{6\brack 1}}-1\right)\leq 1.$$ 
Therefore, the desired result follows.
\end{proof}

From now on, write $$M(n,k,\ell,t)=\max\left\{{k+1\brack t+1}{n-t-1\brack \ell-t-1},\ {\ell+1\brack t+1}{n-t-1\brack k-t-1}\right\}.$$

\begin{lem}\label{2605293}
	Let $\min\{k,\ell\}\geq t+2\geq 4$ and $n\geq k+\ell-t+1+3\delta_{2,q}+\delta_{3,q}$. Then 
	 $$\max\{m_{1}(n,k,\ell,t), m_{1}(n,\ell,k,t) \}>\max\{f_{1}(n,k,\ell,t,t+1,t+1), f_{2}(n,k,\ell,t,t+2,t+2) \}.$$
\end{lem}
\begin{proof}
	We first show $\max\{m_{1}(n,k,\ell,t), m_{1}(n,\ell,k,t) \}>f_{1}(n,k,\ell,t,t+1,t+1)$.
		
	Suppose $q\geq 3$. From Lemma \ref{2604282}, we obtain
	\begin{equation*}
		\begin{aligned}
			\frac{{t+1\brack 1}{\ell-t+1\brack 1}}{{\ell+1\brack t+1}}\leq \frac{q^{\ell+2}}{(q-1)^{2}}\cdot \frac{1}{q^{(t+1)(\ell-t)}}
			=\frac{1}{(q-1)^{2}}\cdot\frac{1}{q^{t(\ell-t-1)-2}}\leq \frac{1}{4}.
		\end{aligned}
	\end{equation*}
	This implies ${t+1\brack 1}{\ell-t+1\brack 1}{n-t-1\brack k-t-1}\leq\frac{1}{4}{\ell+1\brack t+1} {n-t-1\brack k-t-1}$.  By symmetry, we get ${t+1\brack 1}{k-t+1\brack 1}{n-t-1\brack \ell-t-1}<\frac{1}{4}{k+1\brack t+1}{n-t-1\brack \ell-t-1}$. It follows that $f_{1}(n,k,\ell,t,t+1,t+1)\leq\frac{1}{2} M(n,k,\ell,t)$. 
	This together with Lemma \ref{2604191} yields the required result. 
	
	Suppose $q=2$.  Note that $k\geq t+2$. Then
	\begin{equation*}
		\begin{aligned}
			\frac{{k+1\brack t+1}-2{t+1\brack 1}{k-t+1\brack 1}}{2{t+1\brack 1}{k+1\brack t+2}}&=\frac{2^{t+2}-1}{2(2^{t+1}-1)(2^{k-t}-1)}-\frac{{k-t+1\brack 1}}{{k+1\brack k-t-1}}\geq \frac{1}{2^{k-t}-1}-\frac{{k-t+1\brack 1}}{{k-t+3\brack k-t-1}}\\
			&=\frac{1}{2^{k-t}-1}-\frac{315}{(2^{k-t+3}-1)(2^{k-t+2}-1)(2^{k-t}-1)}\\
			&\geq \frac{1}{2^{k-t}-1}-\frac{21}{31(2^{k-t}-1)}=\frac{10}{31(2^{k-t}-1)}.
		\end{aligned}
	\end{equation*}
	Applying Lemma \ref{2604282}, we know $\frac{2^{n-t-1}-1}{2^{\ell-t-1}-1}\geq 2^{k-t+4}\geq \frac{31(2^{k-t}-1)}{10}$.  It follows that 
	$$\frac{\left( {k+1\brack t+1}-2{t+1\brack 1}{k-t+1\brack 1}\right){n-t-1\brack \ell-t-1}}{2{t+1\brack 1}{k+1\brack t+2}{n-t-2\brack \ell-t-2}}=\frac{{k+1\brack t+1}-2{t+1\brack 1}{k-t+1\brack 1}}{2{t+1\brack 1}{k+1\brack t+2}}\cdot \frac{2^{n-t-1}-1}{2^{\ell-t-1}-1}\geq 1.$$
	Therefore, we have 
	\begin{equation*}
		\begin{aligned}
			\frac{m_{2}(n,k,\ell,t)-2{t+1\brack 1}{k-t+1\brack 1}{n-t-1\brack \ell-t-1}}{2{t+1\brack 1}{k+1\brack t+2}{n-t-2\brack \ell-t-2}}
			= \frac{\left( {k+1\brack t+1}-2{t+1\brack 1}{k-t+1\brack 1}\right){n-t-1\brack \ell-t-1}}{2{t+1\brack 1}{k+1\brack t+2}{n-t-2\brack \ell-t-2}}-1\geq 0.
		\end{aligned}
	\end{equation*}
	This implies ${t+1\brack 1}{k-t+1\brack 1}{n-t-1\brack \ell-t-1}\leq \frac{m_{2}(n,k,\ell,t)}{2}$. By symmetry, we get   ${t+1\brack 1}{\ell-t+1\brack 1}{n-t-1\brack k-t-1}\leq \frac{m_{2}(n,\ell,k,t)}{2}$.   Hence 
	$f_{1}(n,k,\ell,t,t+1,t+1)\leq \max\{ m_{2}(n,k,\ell,t),\ m_{2}(n,\ell,k,t) \}$.
	Combining with Lemma \ref{2604191}, we have $f_{1}(n,k,\ell,t,t+1,t+1)< \max\{m_{1}(n,k,\ell,t),m_{1}(n,\ell,k,t)\}$.
	
	Next, we prove $\max\{m_{1}(n,k,\ell,t), m_{1}(n,\ell,k,t) \}>f_{2}(n,k,\ell,t,t+2,t+2)$.  Using $n\geq k+\ell-t+1+3\delta_{2,q}+\delta_{3,q}$ and  Lemma \ref{2604282}, we know 
	\begin{equation*}
		\begin{aligned}
			\frac{{k\brack t}{k-t+1\brack 1}^{2}{n-t-2\brack \ell-t-2}}{{k+1\brack t+1}{n-t-1\brack \ell-t-1}}&= \frac{(q^{t+1}-1)(q^{k-t+1}-1)^{2}(q^{\ell-t-1}-1)}{(q^{k+1}-1)(q-1)^{2}(q^{n-t-1}-1)}\leq \frac{1}{q^{k-t}}\cdot\frac{q^{2k-2t+2}}{(q-1)^{2}}\cdot \frac{1}{q^{n-\ell}}\\
			&=\frac{1}{(q-1)^{2}}\cdot \frac{1}{q^{n-k-\ell+t-2}}\leq \frac{4}{9}.		\end{aligned}
	\end{equation*}
	By symmetry, we have ${\ell\brack t}{\ell-t+1\brack 1}^{2}{n-t-2\brack k-t-2}\leq \frac{4}{9}{\ell+1\brack t+1}{n-t-1\brack k-t-1}$. Hence $f_{2}(n,k,\ell,t,t+2,t+2)\leq \frac{8}{9}M(n,k,\ell,t)$. 		
	 Lemma \ref{2604191} implies  $m_{1}(n,k,\ell,t)>\frac{11}{12}{k+1\brack t+1}{n-t-1\brack \ell-t-1}$ and $m_{1}(n,\ell,k,t)>\frac{11}{12}{\ell+1\brack t+1}{n-t-1\brack k-t-1}$. This implies the desired result. 
\end{proof}

\begin{lem}\label{2605294}
	Let $\ell\geq t+2\geq 4$ and $n\geq \ell+t+2$. Then 
	$$m_{1}(n,t+1,\ell,t)> f_{2}(n,t+1,\ell,t,t+2,t+1).$$
\end{lem}
\begin{proof}
	Note that $f_{2}(n,t+1,\ell,t,t+2,t+1)/ {t+2\brack 1}{n-t-1\brack \ell-t-1}$ is equal to
	\begin{equation*}
		\begin{aligned}
		 &\frac{{\ell\brack t}{\ell-t+1\brack 1}}{ {t+2\brack 1}{n-t-1\brack \ell-t-1}}+\frac{(q^{t+1}-1)(q+1)^{2}(q^{\ell-t-1}-1)}{(q^{t+2}-1)(q^{n-t-1}-1)}\\
			 =&\ \frac{(q^{\ell-t+1}-1)(q^{t+1}-1)}{(q^{\ell-t}-1)(q^{t+2}-1)}\cdot\prod_{i=1}^{\ell-t-1}\frac{q^{\ell-i+1}-1}{q^{n-t-i}-1}+\frac{(q^{t+1}-1)(q+1)^{2}(q^{\ell-t-1}-1)}{(q^{t+2}-1)(q^{n-t-1}-1)}.
		\end{aligned}
	\end{equation*}
	We now estimate the above two terms separately using Lemma \ref{2604282}. 
	If $\ell=t+2$, then the first term is $\frac{(q^{3}-1)(q^{t+1}-1)(q^{\ell}-1)}{(q^{2}-1)(q^{t+2}-1)(q^{n-t-1}-1)}$ and 
	$$\frac{(q^{3}-1)(q^{t+1}-1)(q^{\ell}-1)}{(q^{2}-1)(q^{t+2}-1)(q^{n-t-1}-1)}\leq \frac{q^{3}-1}{q^{2}-1}\cdot\frac{1}{q}\cdot\frac{1}{q^{n-\ell-t-1}}\leq \frac{q^{2}+q+1}{q^{2}(q+1)}=\frac{1}{q+1}+\frac{1}{q^{2}}\leq \frac{7}{12}.$$
	If $\ell\geq t+3$, then for the first term, we get
	\begin{equation*}
		\begin{aligned}
			\frac{(q^{\ell-t+1}-1)(q^{t+1}-1)}{(q^{\ell-t}-1)(q^{t+2}-1)}\cdot\prod_{i=1}^{\ell-t-1}\frac{q^{\ell-i+1}-1}{q^{n-t-i}-1}\leq q\cdot \frac{1}{q^{(\ell-t-1)(n-\ell-t-1)}}\leq \frac{1}{q}\leq \frac{7}{12}.
		\end{aligned}
	\end{equation*}
	For the second term, we always know
	\begin{equation*}
		\begin{aligned}
			\frac{(q^{t+1}-1)(q+1)^{2}(q^{\ell-t-1}-1)}{(q^{t+2}-1)(q^{n-t-1}-1)}\leq \frac{(q+1)^{2}}{q}\cdot\frac{1}{q^{n-\ell}}\leq \frac{(q+1)^{2}}{q^{t+3}}\leq \frac{9}{32}.
		\end{aligned}
	\end{equation*}
	Therefore, we have $f_{2}(n,t+1,\ell,t,t+2,t+1)\leq \frac{83}{96}{t+2\brack 1}{n-t-1\brack \ell-t-1}$. It follows from  Lemma \ref{2604191} that the desired result holds.
\end{proof}

\begin{lem}\label{2605071}
	Let $\min\{k,\ell\}\geq t+2\geq 4$, $n\geq k+\ell-t+1+3\delta_{2,q}+\delta_{3,q}+\delta_{2,q}\delta_{2,t}$. Then 
	$$\max\{m_{1}(n,k,\ell,t),\ m_{1}(n,\ell,k,t)\}>f_{3}(n,k,\ell,t,t+1,t+2, z)$$
	for any $z\in \left\{{s\brack t}{k-t+1\brack 1}+q^{2(s-t+1)}{s\brack t-1}{k-s\brack 1}^{2}: t-1\leq s\leq k-2 \right\}$.
\end{lem}
\begin{proof}
	Pick $t-1\leq s\leq k-3$. By $t\geq 2$ and $k-s\geq 3$, we have
	$$q^{t+1}(q^{k-s-1}-1)^{2}\geq  q^{2k-2s+t-2} +q^{k-s+t+1}-2q^{k-s+t}+q^{t+1}\geq (q^{k-s}-1)^{2}.$$ It follows from Lemma \ref{2604282} that
	\begin{equation*}
		\begin{aligned}
			\frac{q^{2(s-t+2)}{s+1\brack t-1}{k-s-1\brack 1}^{2}}{q^{2(s-t+1)}{s\brack t-1}{k-s\brack 1}^{2}}=\frac{q^{2}(q^{s+1}-1)(q^{k-s-1}-1)^{2}}{(q^{s-t+2}-1)(q^{k-s}-1)^{2}}\geq \frac{q^{t+1}(q^{k-s-1}-1)^{2}}{(q^{k-s}-1)^{2}}\geq 1.
		\end{aligned}
	\end{equation*}
Hence ${s\brack t}{k-t+1\brack 1}+q^{2(s-t+1)}{s\brack t-1}{k-s\brack 1}^{2}$ increases as $s$ increases from $t-1$ to $k-2$.

Note that $f_{3}(n,k,\ell,t,t+1,t+2, z+1)\geq f_{3}(n,k,\ell,t,t+1,t+2,z)$. Let $$z_{0}={k-2\brack t}{k-t+1\brack 1}+q^{2(k-t-1)}(q+1)^{2}{k-2\brack t-1}.$$ It is sufficient to show 
\begin{equation}\label{2605081}
	\begin{aligned}
		f_{3}(n,k,\ell,t,t+1,t+2, z_{0})<\max\{m_{1}(n,k,\ell,t),\ m_{1}(n,\ell,k,t)\}.
	\end{aligned}
\end{equation}

\medskip
\noindent \textbf{Case 1.}  $q=2$ and $t\geq 3$,   or $q\geq 3$.  
\medskip

We now estimate the three terms in $f_{3}(n,k,\ell,t, t+1,t+2,z_{0})$ separately using Lemma \ref{2604282}. For the first term, we have 
\begin{equation*}
	\begin{aligned}
		\frac{{t+1\brack 1}{\ell-t+1\brack 1}^{2}{n-t-2\brack k-t-2}}{{\ell+1\brack t+1}{n-t-1\brack k-t-1}}
		&=\frac{(q^{k-t-1}-1)(q^{t+1}-1)(q^{\ell-t+1}-1)}{(q^{n-t-1}-1)(q-1)^{2}}\cdot\prod_{i=0}^{t-1}\frac{q^{t-i+1}-1}{q^{\ell-i+1}-1}\\
		&\leq \frac{1}{q^{n-k}}\cdot \frac{q^{\ell+2}}{(q-1)^{2}}\cdot\frac{1}{q^{t(\ell-t)}}= \frac{1}{q^{n-k}}\cdot \frac{1}{(q-1)^{2}}\cdot\frac{1}{q^{(t-1)(\ell-t-1)-3}}.
	\end{aligned}
\end{equation*}
For the second term, we know 
\begin{equation*}
	\begin{aligned}
		\frac{{k\brack t}{k-t+1\brack 1}^{2}{n-t-2\brack \ell-t-2}}{{k+1\brack t+1}{n-t-1\brack \ell-t-1}}&=\frac{(q^{t+1}-1)(q^{k-t+1}-1)^{2}(q^{\ell-t-1}-1)}{(q^{k+1}-1)(q-1)^{2}(q^{n-t-1}-1)}\leq \frac{1}{q^{k-t}}\cdot \frac{1}{q^{n-\ell}}\cdot \frac{q^{2k-2t+2}}{(q-1)^{2}}\\
		&=\frac{q^{2}}{(q-1)^{2}}\cdot\frac{1}{q^{n-k-\ell+t}}.
	\end{aligned}
\end{equation*}
Note that  $\frac{q^{k-t}-1}{q^{k+1}-1}\cdot\frac{q^{k-t-1}-1}{q^{k}-1}\cdot\frac{q^{t+1}-1}{q^{k-1}-1}\leq \frac{1}{q^{k+t}}$ and $\frac{q^{k-t}-1}{q^{k+1}-1} \cdot\frac{q^{t+1}-1}{q^{k}-1}\cdot\frac{q^{t}-1}{q^{k-1}-1}\cdot q^{2(k-t-1)}\leq \frac{1}{q^{t+1}}$.  Then
\begin{equation*}
	\begin{aligned}
		&\ \frac{{k-2\brack t}{k-t+1\brack 1}+q^{2(k-t-1)}(q+1)^{2}{k-2\brack t-1}}{{k+1\brack t+1}}\\
		=&\ \frac{(q^{k-t}-1)(q^{k-t-1}-1)(q^{t+1}-1){k-t+1\brack 1}}{(q^{k+1}-1)(q^{k}-1)(q^{k-1}-1)}+\frac{(q^{k-t}-1)(q^{t+1}-1)(q^{t}-1)q^{2(k-t-1)}(q+1)^{2}}{(q^{k+1}-1)(q^{k}-1)(q^{k-1}-1)}\\
		\leq &\ \frac{1}{q^{2t-1}(q-1)}+\frac{(q+1)^{2}}{q^{t+1}}.
	\end{aligned}
\end{equation*}

By the above discussion, we obtain the following inequalities.
If $q= 2$ and $t\geq 3$,  then 
\begin{equation*}
	\begin{aligned}
		f_{3}(n,k,\ell,t, t+1,t+2,z_{0})&\leq \frac{{\ell+1\brack t+1}{n-t-1\brack k-t-1}}{32}+\frac{{k+1\brack t+1}{n-t-1\brack \ell-t-1}}{4}+\frac{19{k+1\brack t+1}{n-t-1\brack\ell-t-1}}{32}\leq \frac{7}{8}M(n,k,\ell,t).
	\end{aligned}
\end{equation*}
If $q= 3$,  then 
\begin{equation*}
	\begin{aligned}
		f_{3}(n,k,\ell,t,t+1,t+2, z_{0})&\leq \frac{{\ell+1\brack t+1}{n-t-1\brack k-t-1}}{36}+\frac{{k+1\brack t+1}{n-t-1\brack \ell-t-1}}{4}+\frac{33{k+1\brack t+1}{n-t-1\brack\ell-t-1}}{54}\leq \frac{8}{9}M(n,k,\ell,t).
	\end{aligned}
\end{equation*}
If $q\geq 4$, then 
\begin{equation*}
	\begin{aligned}
		f_{3}(n,k,\ell,t, t+1,t+2,z_{0})&\leq \frac{{\ell+1\brack t+1}{n-t-1\brack k-t-1}}{36}+\frac{4{k+1\brack t+1}{n-t-1\brack \ell-t-1}}{9}+\frac{19{k+1\brack t+1}{n-t-1\brack\ell-t-1}}{48}\leq \frac{125}{144}M(n,k,\ell,t).
	\end{aligned}
\end{equation*}
Therefore,  Lemma \ref{2604191} implies (\ref{2605081}). The desired result holds.

\medskip
\noindent \textbf{Case 2.} $q= 2$  and $t= 2$.  
\medskip

The similar argument as in Case 1, based on Lemma \ref{2604282}, applies to this case. For the first term, we have 
\begin{equation*}
	\begin{aligned}
		\frac{{3\brack 1}{\ell-1\brack 1}^{2}{n-4\brack k-4}}{{\ell+1\brack 3}{n-3\brack k-3}}&=\frac{147(2^{\ell-1}-1)(2^{k-3}-1)}{(2^{\ell+1}-1)(2^{\ell}-1)(2^{n-3}-1)}\leq \frac{147}{4(2^{\ell}-1)}\cdot\frac{1}{2^{\ell+3}}\leq \frac{49}{1280}\cdot\frac{1}{2^{\ell-3}}.
	\end{aligned}
\end{equation*}
From $2^{n-3}-1\geq 2^{k+\ell}-1\geq 2^{\ell+1}(2^{k-1}-1)$, we obtain 
\begin{equation*}
	\begin{aligned}
		\frac{{k\brack 2}{k-1\brack 1}^{2}{n-4\brack \ell-4}}{{k+1\brack 3}{n-3\brack \ell-3}}&=\frac{7(2^{k-1}-1)^{2}(2^{\ell-3}-1)}{(2^{k+1}-1)(2^{n-3}-1)}\leq 7\cdot \frac{1}{4}\cdot\left(\frac{1}{16}-\frac{1}{2^{\ell+1}}\right)=\frac{7}{64}-\frac{7}{64}\cdot\frac{1}{2^{\ell-3}}.
	\end{aligned}
\end{equation*}
Note that $(2^{k-1}-1)(2^{k}-1)(2^{k+1}-1)\geq 2^{2k+2}(2^{k-2}-1)$. For the third term, we know 
\begin{equation*}
	\begin{aligned}
		\frac{{k-2\brack 2}{k-1\brack 1}+9\cdot2^{2(k-3)}{k-2\brack 1}}{{k+1\brack 3}}
		&=\frac{7(2^{k-2}-1)(2^{k-3}-1)}{(2^{k+1}-1)(2^{k}-1)}+\frac{189\cdot2^{2k-6}(2^{k-2}-1)}{(2^{k+1}-1)(2^{k}-1)(2^{k-1}-1)}\\
		&\leq 7\cdot\frac{1}{8}\cdot\frac{1}{8}+\frac{189}{256}=\frac{217}{256}.
	\end{aligned}
\end{equation*}
Therefore, we get that	$f_{3}(n,k,\ell,2, 3,4,z_{0})$ is no more than
\begin{equation*}
	\begin{aligned}
	 \frac{49{\ell+1\brack 3}{n-3\brack k-3}}{1280\cdot2^{\ell-3}}+\frac{7{k+1\brack 3}{n-3\brack \ell-3}}{64}\cdot\left(1-\frac{1}{2^{\ell-3}}\right)+\frac{217{k+1\brack 3}{n-3\brack\ell-3}}{256}\leq \frac{245}{256}M(n,k,\ell,2).
	\end{aligned}
\end{equation*}
It follows from Lemma \ref{2604191} that (\ref{2605081}) holds. This completes the proof.
\end{proof}

\begin{lem}\label{2605296}
	Let $\ell\geq t+2\geq 4$ and  $n\geq \ell+t+2$. Then 
	$$m_{1}(n,t+1,\ell,t)>f_{3}(n,t+1,\ell,t,t+1,t+1,(q+1)^{2}).$$
\end{lem}

\begin{proof}
	Observe that  $m_{2}(n,t+1,\ell,t)$ minus $f_{3}(n,t+1,\ell,t,t+1,t+1,(q+1)^{2})$ is equal to 
	\begin{equation*}
		\begin{aligned}
			\left( {t+2\brack 1}-(q+1)^{2}\right){n-t-1\brack \ell-t-1}-\left(q^{2}+3q+1\right){t+1\brack 1}{n-t-2\brack \ell-t-2}-{t+1\brack 1}{\ell-t+1\brack 1},
		\end{aligned}
	\end{equation*}
	and this difference is at least $$\left(q^{3}-q \right) {n-t-1\brack \ell-t-1}-\left(q^{2}+3q+1\right){t+1\brack 1}{n-t-2\brack \ell-t-2}-{t+1\brack 1}{\ell-t+1\brack 1}.$$
	For convenience, write the lower bound as $g(n,\ell, t)$. If $\ell=t+2$, then by $n\geq 2t+4$, we have
	\begin{equation*}
		\begin{aligned}
			g(n,t+2,t)&=\frac{q(q^{2}-1)(q^{n-t-1}-1)-2(q+1)^{2}(q^{t+1}-1)}{q-1}\\
			&\geq \frac{q(q^{2}-1)(q^{t+3}-q^{2})-2(q+1)^{2}(q^{t+1}-1)}{q-1}\\
			&= \frac{(q+1)(q^{t+1}-1)(q^{4}-q^{3}-2q-2)}{(q-1)}\geq 0.
		\end{aligned}
	\end{equation*}
	Suppose $\ell\geq t+3$. From Lemma \ref{2604282}, we obtain
	\begin{equation*}
		\begin{aligned}
			g(n,\ell, t)&=\left( \frac{q(q^{2}-1)(q^{n-t-1}-1)}{q^{\ell-t-1}-1}-(q^{2}+3q+1){t+1\brack 1}\right){n-t-2\brack \ell-t-2}-{t+1\brack 1}{\ell-t+1\brack 1}\\
			&\geq \left( q^{t+3}(q^{2}-1)-q^{t+1}(q^{2}+3q+1)\right){n-t-2\brack \ell-t-2}-{t+1\brack 1}{\ell-t+1\brack 1}\\
			&=q^{t+1}(q^{4}-2q^{2}-3q-1){n-t-2\brack \ell-t-2}-{t+1\brack 1}{\ell-t+1\brack 1}.
		\end{aligned}
	\end{equation*}
	Note that $q^{4}-2q^{2}-3q-1\geq1$ and $${n-t-2\brack \ell-t-2}\geq q^{(\ell-t-2)(n-\ell)}\geq q^{4(\ell-t-2)}=q^{(\ell-t+1)+3(\ell-t)-9}\geq q^{\ell-t+1}.$$
	We further conclude $g(n,\ell, t)\geq 0$. In summary, we have $$m_{2}(n,t+1,\ell,t)\geq f_{3}(n,t+1,\ell,t,t+1,t+1,(q+1)^{2}).$$ It follows from Lemma \ref{2604191} that $f_{3}(n,t+1,\ell,t,t+1,t+1,(q+1)^{2})<m_{1}(n,t+1,\ell, t)$.
\end{proof}

\begin{lem}\label{2605312}
	Let $\ell\geq k+2\geq t+3\geq 5$ and $n\geq k+\ell+\lceil\frac{t+1}{\ell-k-1}\rceil$. 
	\begin{itemize}
		\item[\normalfont(i)]	$m_{1}(n,k,\ell,t)>{\ell\brack t}{n-t\brack k-t}+{k-t+1\brack 1}{n-t-1\brack \ell-t-1}$.
		\item[\normalfont(ii)]  $m_{1}(n,k,\ell,t)>m_{1}(n,\ell,k,t)$.
	\end{itemize}
\end{lem}
\begin{proof} 
	From Lemma \ref{2604191}, we obtain 
	\begin{equation}\label{2606011}
	m_{1}(n,k,\ell,t)>\frac{7}{8}{k+1\brack t+1}{n-t-1\brack \ell-t-1}.
	\end{equation}

 (i)	We estimate ${\ell\brack t}{n-t\brack k-t}$ and ${k-t+1\brack 1}{n-t-1\brack \ell-t-1}$ separately using Lemma \ref{2604282}. 	
	For the first term, we have 
	\begin{equation*}
		\begin{aligned}
			\frac{{\ell\brack t}{n-t\brack k-t}}{{k+1\brack t+1}{n-t-1\brack\ell-t-1}}&=\frac{q^{n-t}-1}{(q^{\ell-t}-1)(q^{n-\ell+1}-1)}\cdot(q^{t+1}-1)\cdot\prod_{i=1}^{\ell-k-1}\frac{q^{\ell-i+1}-1}{q^{n-k-i+1}-1}\\
			&\leq \frac{1}{q(1-\frac{1}{q^{\ell-t}})(1-\frac{1}{q^{n-\ell+1}})}\cdot \frac{q^{t+1}-1}{q^{(n-k-\ell)(\ell-k-1)}}\leq \frac{128}{217}.
		\end{aligned}
	\end{equation*}
	For the second term, we know 
	$\frac{{k-t+1\brack 1}}{{k+1\brack t+1}}=\prod_{i=1}^{t}\frac{q^{t-i+2}-1}{q^{k-i+2}-1}\leq \frac{1}{q^{(k-t)t}}\leq \frac{1}{4}$. 
	 By (\ref{2606011}), we get (i).
	
	(ii) By the definition of $m_{1}(n,\ell,k,t)$, we know $m_{1}(n,\ell,k,t)\leq {\ell+1\brack t+1}{n-t-1\brack k-t-1}+{\ell+1\brack 1}$.  From Lemma \ref{2604282}, we obtain
	\begin{equation*}
		\begin{aligned}
			\frac{m_{1}(n,\ell,k,t)}{{k+1\brack t+1}{n-t-1\brack \ell-t-1}}&\leq \left(\prod_{i=1}^{t+1}\frac{q^{\ell-i+2}-1}{q^{k-i+2}-1}\right)\left(\prod_{i=1}^{\ell-k}\frac{q^{\ell-t-i}-1}{q^{n-k-i+1}-1}\right)+\frac{{\ell+1\brack 1}}{{k+1\brack t+1}{n-t-1\brack \ell-t-1}}\\
			&\leq  \frac{q^{(\ell-k+1)(t+1)}}{q^{(n-k-\ell+t+1)(\ell-k)}}+\frac{q^{\ell+1}}{q^{(t+1)(k-t)+(\ell-t-1)(n-\ell)}}\\
			&=\frac{q^{t+1}}{q^{(n-k-\ell)(\ell-k-1)+n-k-\ell}}+ \frac{q}{q^{(t+1)(k-t-1)+(\ell-t-1)(n-\ell-1)}}\leq \frac{1}{q}+\frac{1}{q^{5}}.
		\end{aligned}
	\end{equation*}
	It follows from (\ref{2606011}) that (ii) holds.	
\end{proof}

\begin{lem}\label{2604221}
	Let $k\geq 2$.  For any $0\leq x,y\leq k-1$ with $x+y\leq k$, we have
	$${x\brack 1}{k\brack 1}+{y\brack 1}{k\brack 1}+{k\brack 1}-{x+y\brack 1}\leq {k-1\brack 1}\left({k\brack 1}+q\right)+1.$$
	Moreover, equality holds if and only if $(x,y)=(1,k-1)$ or $(x,y)=(k-1,1)$.
\end{lem}
\begin{proof}
	One can check that the LHS is equal to 
	${k\brack 1}\frac{q^{x}+q^{y}-2}{q-1}+{k\brack 1}-{x+y\brack 1}$. 
	If $x+y\leq k-1$, then by $q^{x}+q^{y}\leq q^{x+y}+1$ for any $0\leq x,y\leq k-1$, the LHS is no more than
	\begin{equation*}
		{k\brack 1}\frac{q^{x+y}-1}{q-1}+{k\brack 1}-{x+y\brack 1}
		=q{x+y\brack 1}{k-1\brack 1}+{k\brack 1}\leq {k-1\brack 1}\left({k\brack 1}+q\right).
	\end{equation*}
	If  $x+y=k$, then  $x\geq 1$. Note that $q^{x}+q^{k-x}\leq q^{k-1}+q$ for any $1\leq x\leq k-1$, and equality holds if and only if $x=1$ or $x=k-1$.  Hence the LHS is no more than $ {k\brack 1}\frac{q^{k-1}+q-2}{q-1} ={k-1\brack 1}\left({k\brack 1}+q\right)+1$, and equality holds if and only if $(x,y)=(1,k-1)$ or $(x,y)=(k-1,1)$. The desired result follows.
\end{proof}

\begin{lem}\label{2605292}
	Let $k\geq s+2$. For any $0\leq w\leq s-1$ and $s\leq x,y\leq k-1$ with $x+y-w\leq k$, the function  
	$$q^{w+s}{k-s\brack 1}\left({x-w\brack 1}+{y-w\brack 1}\right)+q^{s}\left( {k\brack 1}-{x\brack 1}-{y\brack 1}+{w\brack 1}\right)$$
	is no more than $q^{2s-1}{k-s\brack 1}\left({k-s\brack 1}+q \right)$.
\end{lem} 
\begin{proof}
	In the remaining of this proof, write the function as $g(k,s,w,x,y)$. Then 
	$$g(k,s,w,x,y)= q^{w+s}{k-w\brack 1}+q^{w+s}\left({k-s\brack 1}-1\right)\left(\frac{q^{x-w}+q^{y-w}-2}{q-1}\right).$$
 Note that $(x-w)+(y-w)\leq k-w$ and  $s-w\leq x-w\leq k-s$.  We  conclude   
	$$q^{x-w}+q^{y-w}\leq q^{x-w}+q^{(k-w)-(x-w)}\leq q^{k-s}+q^{s-w}. $$
This together with ${k-w\brack 1}=q^{k-s}{s-w\brack 1}+{k-s\brack 1}$ yields 
	\begin{equation*}
			\begin{aligned}
			g(k,s,w,x,y)&\leq q^{w+s}{k-w\brack 1}+q^{w+s}\left({k-s\brack 1}-1\right)\left({k-s\brack 1}+{s-w\brack 1}\right)\\
			&=q^{s}{k-s\brack 1}\left( q^{w}{k-s\brack 1}+q^{w+1}{s-w\brack 1}\right).
		\end{aligned}
	\end{equation*}
It is sufficient to show $q^{w}{k-s\brack 1}+q^{w+1}{s-w\brack 1}\leq q^{s-1}{k-s\brack 1}+q^{s}$. 
	
	If $s=1$, then the desired result is clear. Suppose  $s\geq 2$. For any $0\leq w\leq s-2$, we have
	\begin{equation*}
		\begin{aligned}
			q^{w+1}{k-s\brack 1}+q^{w+2}{s-w-1\brack 1}-q^{w}{k-s\brack 1}-q^{w+1}{s-w\brack 1}
			= q^{w}(q^{k-s}-q-1)\geq 0.
		\end{aligned}
	\end{equation*}
	It follows that $q^{w}{k-s\brack 1}+q^{w+1}{s-w\brack 1}$ attains maximum value if $w=s-1$, as required. 
\end{proof}

\medskip
\noindent{\bf Conflict of interest.}	
We have no known financial and personal relationships with other people or organizations that could potentially 
influence the work in this paper.

\medskip
\noindent{\bf Acknowledgment.}	
K. Wang is supported by the National Natural Science Foundation of \\China (12131011, 12571347) and Beijing Natural Science Foundation (1252010, 1262010). T. Yao is supported by   Natural Science Foundation of Henan (262300422621).

\medskip
\noindent{\bf Data availability.}	
No data was used for the research described in this paper.

\end{document}